%% file: main.tex
\documentclass[a4paper, 11pt, parskip=half]{amsart}

\usepackage{custom-template, shortcuts}

\makeatletter
\@namedef{subjclassname@2020}{\textup{2020} Mathematics Subject Classification}
\makeatother

\begin{document}

\title{On the images of Galois representations attached to low weight Siegel modular forms}
\date{}
\author{Ariel Weiss}
\address{Ariel Weiss, Department of Mathematics, Ben-Gurion University of the Negev, Be'er Sheva 8410501, Israel.\vspace*{-3pt}}
\email{weiss.ariel@mail.huji.ac.il}

\begin{abstract}
		Let $\pi$ be a cuspidal automorphic representation of $\Gf(\AQ)$, whose archimedean component is a holomorphic discrete series or limit of discrete series representation. If $\pi$ is not CAP or endoscopic, then we show that its associated $\l$-adic Galois representations are irreducible and crystalline for $100\%$ of primes $\l$. If, moreover, $\pi$ is neither an automorphic induction nor a symmetric cube lift, then we show that, for $100\%$ of primes $\l$, the image of its mod $\l$ Galois representation contains $\Sp_4(\Fl)$.
\end{abstract}

\maketitle
\input{introduction}
\input{Gal-Reps}

\input{partial-irred}
\input{crystallinity}

\input{full-irred}
\input{residual-irred}

\bibliography{bibliography}
\bibliographystyle{alpha}

\end{document}

%% file: introduction.tex
\section{Introduction}

Under the Langlands correspondence, where automorphic representations of $\GL_n$ should correspond to $n$-dimensional Galois representations, \emph{cuspidal} automorphic representations should correspond to \emph{irreducible} Galois representations. More generally, one expects that the image of an automorphic Galois representation should be as large as possible, unless there is an automorphic reason for it to be small.

In this paper, we study the images of Galois representations attached to low weight, genus $2$ Siegel modular forms. These automorphic forms are the genus $2$ analogues of weight $1$ modular forms, and are of particular interest due to their conjectural relationship with abelian surfaces. Our main result is the following theorem:

\begin{theorem}\label{mainthm}
Let $\pi$ be a cuspidal automorphic representation of $\Gf(\AQ)$ such that $\pi_\infty$ is a holomorphic discrete series or limit of discrete series representation. Assume that the weak functorial lift of $\pi$ to $\GL_4$ exists and is cuspidal. Let $E$ be the coefficient field of $\pi$. For each prime $\lambda$ of $E$, of residue characteristic $\l$, let
\[\rho_\lambda\:\Ga\Q\to\Gf(\overline{E}_\lambda)\]
be the $\lambda$-adic Galois representation associated to $\pi$. Then:
\begin{enumerate}
\item If $\rho_\lambda|_{\Ql}$ is de Rham, and if $\l\ge 5$, then $\rho_\lambda$ is irreducible.
\item $\rho_\lambda|_{\Ql}$ is crystalline for all $\lambda\mid\l$ for a set of primes $\l$ of Dirichlet density $1$.
\end{enumerate}
In particular, $\rho_\lambda$ is irreducible for all $\lambda\mid\l$ for a set of primes $\l$ of Dirichlet density $1$.
\end{theorem}

The automorphic representations $\pi$ in the statement of the theorem are exactly those that arise from classical genus $2$ vector-valued Siegel modular forms. The assumption that the functorial lift is cuspidal amounts to demanding that $\pi$ arises from a Siegel modular form that is not CAP or endoscopic; in these cases, the associated Galois representation is known to be reducible. In general, the existence of this lift follows from Arthur's classification \cite{Arthur2013}; see \Cref{arthur} for further discussion.

Beyond irreducibility, we also prove a big image theorem for the images of the mod $\lambda$ representations attached to $\pi$:

\begin{theorem}\label{residual-irred-intro}
	Let $\pi$ be a cuspidal automorphic representation of $\Gf(\AQ)$ such that $\pi_\infty$ is a holomorphic discrete series or limit of discrete series representation. Assume that the weak functorial lift of $\pi$ to $\GL_4$ exists and is cuspidal. Let $E$ be the coefficient field of $\pi$. For each prime $\lambda$ of $E$, of residue characteristic $\l$, let $\F_\lambda = \O_E/\lambda$ and let
	\[\orho_{\lambda}\: \Ga \Q \to \Gf(\F_\lambda)\]
	be the mod $\lambda$ Galois representation associated to $\pi$. Let $\LL$ be the set of primes $\lambda$ of $E$ for which $\rho_\lambda|_{\Ql}$ is crystalline. Then:
	\begin{enumerate}
		\item For all but finitely many primes $\lambda\in \LL$, $\orho_{\lambda}$ is irreducible.
		\item If $\pi$ is neither an automorphic induction nor a symmetric cube lift, then, for all but finitely many primes $\lambda\in \LL$, the image of $\orho_{\lambda}$ contains $\Sp_4(\Fl)$.
	\end{enumerate}
\end{theorem}

The corresponding results for elliptic modular forms were proven by Deligne--Serre, Ribet and Momose \cites{Deligne-Serre, Ribet77, momose,Ribet85}. For high weight Siegel modular forms, irreducibility for all but finitely many primes follows from the work of Ramakrishnan \cite{Ramakrishnan}, while the analogue of \Cref{residual-irred-intro}  follows from \cite[Prop.\ 5.3.2]{BLGGT}  and from the work of Dieulefait--Zenteno \cite{DZ}, but only for a set of primes $\lambda$ of residual Dirichlet density $1$. Conjecturally, $\rho_\lambda$ should be irreducible for all primes $\lambda$, and $\orho_\lambda$ should be irreducible for all but finitely many primes $\lambda$. Since, in the high weight case, the set $\LL$ contains all but finitely many primes, \Cref{residual-irred-intro} gives an improvement on existing results even in the cohomological case.

These previous results depend crucially on the facts that the associated Galois representations are geometric, satisfy the Ramanujan conjectures and, in the high weight Siegel modular form case, are Hodge--Tate regular. All other recent results proving the irreducibility of automorphic Galois representations rely on these inputs \cite{Ramakrishnan, BLGGT,CG, patrikis-taylor, Xia}. 

The novelty of this paper is that we prove a big image theorem in a situation where these key inputs are not available. In the case of low weight Siegel modular forms, the Hodge--Tate--Sen weights of $\rho_\lambda$ are irregular, purity is an open problem and, a priori, crystallinity is not known. Indeed, a priori, we do not even know that $\rho_\lambda$ is Hodge--Tate.

\subsection{Methods}

Our proof of \Cref{mainthm} proceeds in two steps. First, in \Cref{thm:partial-irred}, we prove, with no assumptions on the prime $\lambda$, that either $\rho_\lambda$ is irreducible or it decomposes as a direct sum of two irreducible two-dimensional representations that are Hodge--Tate regular and odd. Our key input is \Cref{thm:mock-ramanujan}, which demonstrates that, in many cases, the Jacquet--Shalika bounds \cite{jacquetshalika1}, can be used to rule out the existence of certain subrepresentations. These bounds are weaker than those predicted by the generalised Ramanujan conjecture, and are known in general for cuspidal automorphic representations of $\GL_n$,

If these two-dimensional subrepresentations of $\rho_\lambda$ were modular, then a routine $L$-functions argument would lead to a contradiction. By recent work of Pan \cite{pan}, these representations are modular if they are de Rham and if $\l\ge 5$, but, a priori, these representations need not even be Hodge--Tate. On the other hand, a criterion of Jorza \cite{jorza2012} shows that, if $\pi$ is unramified at $\l$ and if $\lambda\mid\l$, then $\rho_\lambda$ is crystalline if the four Satake parameters of $\pi_\l$ are distinct. 

Our second step is to observe that Jorza's criterion can be translated into a condition on the image of $\rho_\lambda$ for a single prime $\lambda$. Indeed, for all unramified primes $p\ne \l$, the Satake parameters of $\pi_p$  are (up to normalisation) exactly the eigenvalues $\rho_\lambda(\Frob_p)$. Combining work of Rajan \cite{Rajan} with the restrictions on the decomposition of $\rho_\lambda$ proven in \Cref{thm:partial-irred}, we prove that the characteristic polynomial of $\rho_\lambda(\Frob_p)$ has distinct roots for a set of primes $p$ of Dirichlet density $1$. Thus, $\rho_\lambda$ is crystalline for all $\lambda\mid\l$ for a set of primes $\l$ of Dirichlet density $1$, and we can apply Pan's result to reach a contradiction.

This distinctness of Satake parameters is a key input in the proof of \Cref{residual-irred-intro}. Indeed, in \Cref{lem:no-even}, we use it to prove that $\orho_\lambda$ cannot have an even two-dimensional subrepresentation for infinitely many $\lambda\in\LL$. In the cohomological case, this argument allows us to strengthen the result of \cite{DZ} to apply to all but finitely many primes, rather than just a density $1$ set of primes.

\subsection{The structure of this paper}

In \Cref{sec:autreps}, we review key properties of automorphic representations of $\Gf(\AQ)$. In particular, we define holomorphic (limit of) discrete series representations in terms of their $L$-parameters, and we discuss the transfer map from $\Gf$ to $\GL_4$. In \Cref{galreps-section}, we survey existing results on the construction of Galois representations associated to Siegel modular forms. In addition, we prove that, for low weight forms, the Galois representations are symplectic.

In \Cref{irred-section}, we prove \Cref{thm:partial-irred} and the key input \Cref{thm:mock-ramanujan}. In \Cref{distinctness}, we use the results of \Cref{irred-section} to prove that the Satake parameters of $\pi_p$ are distinct for a set of primes $p$ of Dirichlet density $1$. Applying a theorem of Jorza \cite{jorza2012}, we deduce that $\rho_\lambda$ is crystalline for all primes $\lambda\mid\l$ for a set of primes $\l$ of Dirichlet density $1$. Using this result, in \Cref{full-irred}, we apply a recent result of Pan \cite{pan} to complete the proof of \Cref{mainthm}. Finally, in \Cref{residual-irred-section}, we prove \Cref{residual-irred-intro}. 

\section{Automorphic representations of \texorpdfstring{$\Gf(\AQ)$}{GSp(4, A)}}\label{sec:autreps}

In this section, we review aspects of the local and global theory of automorphic representations of $\Gf(\AQ)$.

\begin{definition}
	
For a ring $R$, let
\[\Gf(R) = \set{\gamma\in\GL_4(R): \gamma^t J \gamma = \nu J,\ \nu \in R\t},\]
where  $J = \br{\begin{smallmatrix}
	0 & 0 & 0 & 1\\
	0 & 0 & 1 & 0\\
	0 & -1 & 0 & 0\\
	-1 & 0 & 0 & 0
\end{smallmatrix}}$. For $\gamma\in \Gf(R)$, the constant $\nu$ is called the \emph{similitude} of $\gamma$ and is denoted $\simil(\gamma)$. Let $\Sp_4(R)$ be the subgroup of elements for which $\simil(\gamma) = 1$.
\end{definition}

\subsection{Archimedean $L$-parameters}\label{background-archimedean}

If $\pi$ is a cuspidal automorphic representation of $\Gf(\AQ)$, then $\pi$ can be decomposed as a restricted tensor product $\pi = \tensor'_p\pi_p$, with each $\pi_p$ an admissible representation of $\Gf(\Qp)$. Throughout this paper, $\pi$ will denote a unitary cuspidal automorphic representation of $\Gf(\AQ)$ whose archimedean component $\pi_\infty$ is a holomorphic discrete series or limit of discrete series representation. In this subsection, we define these representations via their associated $L$-parameters. Our exposition follows \cite[\S3.1]{mok2014galois}.

Let $W_{\C}=\C\t$ denote the Weil group of $\C$ and let $W_\R$ be the Weil group of $\R$. Then
\[W_\R = \C\t\sqcup\C\t j,\]
where $j^2=-1$ and $jzj\ii=\zb$ for $z\iC\t$. We define a map 
\[|\cdot|\:W_\R\to \R\t\]
by sending $z\iC\t$ to $z\overline z = |z|^2$ and $j\mapsto -1$. This map induces an isomorphism $W_\R^{ab}\cong\R\t$. 
%Note that, under this isomorphism, the usual norm map
%\begin{align*}
%	|\cdot|\:\R\t&\to\R\t\\
%	x&\mapsto |x|
%\end{align*}
%corresponds to the map $|\cdot|^2\:W_\R\to\R\t$. 

\subsubsection{$L$-parameters for \texorpdfstring{$\GL_2(\R)$}{GL(2,R)}}

We begin by defining the $L$-parameters that correspond to the discrete series and limit of discrete series representations of $\GL_2(\R)$. Let $w, n$ be integers, with $n\ge 0$ and ${n \equiv w+1\pmod 2}$. Define
\begin{align*}
	\phi_{(w;\,n)}\:W_\R&\to\GL_2(\C)\\
	z&\mapsto|z|^{-w}\br{\begin{smallmatrix}
		(z/\zb)^{n/2}\\&(z/\zb)^{-n/2}
	\end{smallmatrix}}\quad\text{for }z\iC\t\\
	j&\mapsto\br{\begin{smallmatrix}
		&1\\(-1)^n
	\end{smallmatrix}}.
\end{align*}

Here, if $z=re^{i\theta}\iC\t$, then we let $(z/\overline{z})^{n/2}=e^{in\theta}$. When $n\ge 1$, $\phi_{(w;\,n)}$ corresponds to the weight $n+1$ discrete series representation of $\GL_2(\R)$ with central character $a\mapsto a^{-w}$, for $a\iR\t$. When $n=0$, $\phi_{(w;\,n)}$ corresponds to the limit of discrete series representation of $\GL_2(\R)$ with central character $a\mapsto a^{-w}$. The parity condition on $w$ ensures that, in both cases, $\phi_{(w;\,n)}$ is the archimedean $L$-parameter attached to a classical modular form of weight $n+1$.

\subsubsection{$L$-parameters for \texorpdfstring{$\Gf(\R)$}{GSp(4, R)}}\label{background-l-params} 

Next, we define the $L$-parameters whose $L$-packets contain the (limit of) discrete series representations of $\Gf(\R)$. For integers $w, m_1, m_2$, with $m_1>m_2\ge 0$ and ${m_1+m_2\equiv w+1\pmod 2}$, we define an $L$-parameter 
\[\phi_{(w;\, m_1,m_2)}\:W_\R\to\Gf(\C)\]
by
\begin{align*}
	z&\mapsto |z|^{-w}\sdmat{(z/\zb)^{(m_1+m_2)/2}}{(z/\zb)^{(m_1-m_2)/2}}{(z/\zb)^{-(m_1-m_2)/2}}{(z/\zb)^{-(m_1+m_2)/2}}
\end{align*}
for $z\iC\t$, and
\begin{align*}
	j&\mapsto \br{\begin{smallmatrix}
		&&&1\\
		&&1&\\
		&(-1)^{m_1+m_2}&&\\
		(-1)^{m_1+m_2}&&&\end{smallmatrix}}.
\end{align*}

The image of $\phi_{(w;\, m_1,m_2)}$ lies in $\Gf(\C)$ and has similitude character given by
\begin{align*}
	z&\mapsto |z|^{-2w}\\
	j&\mapsto (-1)^{w} = (-1)^{m_1 + m_2+1}.
\end{align*}
If we compose $\phi_{(w;\, m_1,m_2)}$ with the inclusion $\Gf(\C)\hookrightarrow\GL_4(\C)$, the resulting representation of $W_\R$ is isomorphic to the direct sum $\phi_{(w;\, m_1+m_2)}\+\phi_{(w;\, m_1-m_2)}$ of parameters of discrete series representations of $\GL_2(\R)$.

The $L$-packet corresponding to $\phi_{(w;\, m_1,m_2)}$ has two elements:
\[\set{\pi^H_{(w;\, m_1,m_2)}, \pi^W_{(w;\, m_1,m_2)}}.\]
Both $\pi^H_{(w;\, m_1,m_2)}$ and $\pi^W_{(w;\, m_1,m_2)}$ have central character given by $a\mapsto a^{-w}$ for ${a\in \R\t}$. When $m_2\ge 1$ they are (up to twist) discrete series representations: $\pi^H_{(w;\, m_1,m_2)}$ is a holomorphic discrete series and $\pi^W_{(w;\, m_1,m_2)}$ is a generic discrete series representation.  When $m_2 = 0$, $\pi^H_{(w;\, m_1,0)}$ is a holomorphic limit of discrete series and $\pi^W_{(w;\, m_1,0)}$ is a generic limit of discrete series representation.

\subsection{The ``weight'' of a holomorphic (limit of) discrete series representation}\label{sec:weight}

The Blattner parameter $(k_1, k_2)$  of the holomorphic (limit of) discrete series representation $\pi^H_{(w;\, m_1,m_2)}$ is defined to be
\[k_1 = m_1 + 1,\quad k_2 = m_2 +2.\]
If $\pi$ is the automorphic representation corresponding to a classical Siegel modular form of weight $(k_1, k_2)$---i.e.\ a vector-valued Siegel modular form with weight $\Sym^{k_{1} - k_{2}}\det^{k_{2}}$---then $\pi_\infty$ is a holomorphic (limit of) discrete series representation with Blattner parameter $(k_{1}, k_{2})$. Hence, if $\pi$ is a cuspidal automorphic representation of $\Gf(\AQ)$ such that $\pi_\infty$ is a holomorphic (limit of) discrete series representation, then we refer to the Blattner parameter $(k_{1}, k_{2})$ of $\pi_\infty$ as the weight of $\pi$.

We say that $\pi$ is \emph{cohomological} or has \emph{high weight} if $k_{2}\ge 3$. If $k_{2}= 2$, then we say that $\pi$ is \emph{non-cohomological} or has \emph{low weight}. Cohomological automorphic representations can be realised in the \'etale cohomology of a local system on a suitable Shimura variety, while non-cohomological automorphic representations can only be realised in coherent cohomology.

\subsection{Other non-degenerate limits of discrete series for \texorpdfstring{$\Gf(\R)$}{GSp(4, R)}}

Finally, for completeness, we note that there is another $L$-packet of non-degenerate limits of discrete series. Following \cite{SchmidtHodge}, if $m$ is a positive integer, then this $L$-packet corresponds, up to twist, to the $L$-parameter $W_\R\to\Gf(\C)$ that sends
\begin{equation*}
	z\mapsto \sdmat{(z/\zb)^{m/2}}{1}{1}{(z/\zb)^{-m/2}},\qquad\qquad
	j\mapsto\br{\begin{smallmatrix}
		&&&1\\
		&&1&\\
		&1&&\\
		1&&&
	\end{smallmatrix}}.
\end{equation*}

If $\pi$ is a cuspidal automorphic representation of $\Gf(\AF)$, whose archimedean components are all either discrete series or non-degenerate limits of discrete series, then, by \cite[Thm.\ 10.5.1]{GoldringKoskivirta}, we can attach Galois representations to $\pi$. Assuming local-global compatibility results for these Galois representations as in \Cref{non-cohom-gal-rep}, it is likely that the arguments of this paper can be extended to prove a big image theorem for these Galois representations. 

\subsection{The transfer map}\label{arthur}

The principle of Langlands functoriality predicts that there is a global transfer map from automorphic representations of $\Gf$ to automorphic representations of $\GL_4$. Indeed, ${}^L\!\Gf = \Gf(\C)$ and ${}^L\!\GL_4 = \GL_4(\C)$, and the existence of the embedding $\Gf(\C)\hookrightarrow\GL_4(\C)$ of $L$-groups means that there should be a corresponding lifting of automorphic representations. This lifting should be compatible with the local Langlands correspondence: if an automorphic representation $\pi$ of $\Gf(\AQ)$ lifts to an automorphic representation $\Pi$ of $\GL_4(\AQ)$, then, for almost all primes $p$, the Weil--Deligne representation
\[W_{\Qp}\times\SL_2(\C)\to \GL_4(\C)\]
corresponding to $\Pi_p$ via the local Langlands correspondence for $\GL_4$ should be isomorphic to the Weil--Deligne representation
\[W_{\Qp}\times\SL_2(\C)\to \Gf(\C)\hookrightarrow\GL_4(\C)\]
corresponding to $\pi_p$ via the local Langlands correspondence for $\Gf$.
Moreover, an automorphic representation $\Pi$ should be in the image of this lifting if and only if it is of \emph{symplectic type}, i.e.\ if there is a Hecke character $\chi$ such that $\Pi\cong\Pi\dual\tensor\chi$ and the partial $L$-function $L^*(\ext\Pi\tensor\chi\ii, s)$ (which exists by \cite{kim2003functoriality}) has a pole at $s=1$.

This lifting has been achieved for \emph{globally generic} cuspidal automorphic representations of $\Gf(\AQ)$ by Asgari--Shahidi \cite{asgari-shahidi}. However, the automorphic representations that correspond to Siegel modular forms are not globally generic: their archimedean components are holomorphic (limit of) discrete series representations, which are not generic.

If $\pi$ is cohomological, then Weissauer \cite{Weissauersymplectic}*{Thm.\ 1.1} has shown that $\pi$ is weakly equivalent to a globally generic automorphic representation of $\Gf(\AQ)$. Combined with the result of Asgari--Shahidi, we obtain a weak lift of $\pi$ to $\GL_4$. In particular, the results of this paper apply unconditionally to cohomological Siegel modular forms.

More generally, the existence of a lift to $\GL_4$ follows from Arthur's endoscopic classification \cite{Arthur2013}. Indeed, the case of $\PGSp_4$, i.e.\ where $\pi$ has trivial central character, has been spelled out in detail in \cite{Arthur2004}, and this work has been generalised by Gee--Ta\"ibi \cite{geetaibi} to cover all automorphic representations of $\Gf$. However, these works are all dependent on several papers that have been announced, but have yet to appear: see the discussions in \cite{geetaibi}*{pp.\ 3} and in \cite{mok2014galois}*{pp.\ 527}.

%% file: Gal-Reps.tex
\section{Galois representations associated to Siegel modular forms}\label{galreps-section}

In this section, we review the construction of Galois representations associated to Siegel modular forms. In Sections \ref{section-galreps-cohom} and \ref{section-galreps-noncohom}, we discuss the construction of Galois representations in the cohomological and non-cohomological cases. In \Cref{section-gsp4-valued}, we prove that, in the non-cohomological case, the Galois representations are symplectic.

\subsection{The case of cohomological weight}\label{section-galreps-cohom}

We review the construction of Galois representations attached to high weight Siegel modular forms.

\begin{theorem}\label{cohom-gal-rep}
	Let $\pi$ be a cuspidal automorphic representation of $\Gf(\AQ)$ of weight $(k_1, k_2)$, $k_1\ge k_2\ge 3$. Let $S$ denote the set of primes at which $\pi$ is ramified. Let $E$ denote the coefficient field of $\pi$. Then for every place $\lambda$ of $E$, of residue characteristic $\l$, there exists a continuous, semisimple, symplectic Galois representation
	\[\rho_\lambda\:\Ga\Q\to\Gf(\overline E_\lambda)\]
	that satisfies the following properties:
	\begin{enumerate}
		\item The representation is unramified at all primes $p\notin S\cup\{\l\}$.
				\item The similitude character $\simil\rho_\lambda$ is odd and $\rho_\lambda\simeq\rho_\lambda\dual\tensor\simil\rho_\lambda,$	where $\rho_\lambda\dual$ is the dual representation.
		\item The representation $\rho_\lambda|_{\Ql}$ is de Rham for all primes $\lambda$, and crystalline if $\l\notin S$.
		\item Local-global compatibility is satisfied up to semisimplification: for any prime $p\ne \l$, 
		\[\mathrm{WD}(\restr{\rho_\lambda}{\Qp})^{ss}\cong\mathrm{rec}_p(\pi_p\tensor|\simil|_p^{(3-k_1-k_2)/2})^{ss},\]
		where $\mathrm{rec}_p$ denotes the local Langlands reciprocity map \cite{gan2011local}. 
		\item The set of Hodge--Tate weights of $\rho_\lambda|_{\Ql}$ is $\set{0,  k_2-2,k_1-1, k_1+k_2-3}.$
	\end{enumerate}

	Moreover, if $\pi$ is not CAP $($i.e.\ of Saito--Kurokawa type$)$, then $\rho_\lambda$ satisfies the following stronger properties:
	\begin{enumerate}
	\setcounter{enumi}{6}
		\item Local-global compatibility is satisfied up to Frobenius semisimplification: for any prime $p\ne \l$, 
		\[\mathrm{WD}(\restr{\rho_\lambda}{\Qp})^{F\text{-}ss}\cong\mathrm{rec}_p(\pi_p\tensor|\simil|_p^{(3-k_1-k_2)/2}).\]
		
		\item The representation $\rho_\lambda$ is pure of weight $k_1 + k_2 - 3$. In particular, if $p\notin S\cup \{\l\}$ and if $\alpha\iC$ is a root of the characteristic polynomial of $\rho_\lambda(\Frob_p)$, then $|\alpha| = p^{\frac{k_1+k_2-3}2}$.
	\end{enumerate}

	Finally, if $\pi$ is neither CAP nor endoscopic and if $\l\ge 5$, then $\rho_\lambda$ is irreducible.
	
\end{theorem}

\begin{remark}
	If $p\notin S\cup\{\l\}$, $a_p$ is the (suitably normalised) eigenvalue of the Hecke operator $T_p$ and $\chi$ is the Galois character associated to the central character of $\pi$, then condition $(iv)$ implies that
	\[\Tr\rho_\lambda(\Frob_p)= a_p,\qquad
	\simil\rho_\lambda = \chi\epsilon_{\l}^{k_1+k_2-3},\]
	where $\epsilon_\l$ is the $\l$-adic cyclotomic character.
\end{remark}

\begin{proof}
When $\pi$ is CAP or endoscopic, the construction of $\rho_\lambda$ follows from the construction of Galois representations for elliptic modular forms. A discussion of these cases is given in \cite[pp. 537--538]{mok2014galois}. In these cases, $\rho_\lambda$ is reducible.

When $\pi$ is of general type (type \textbf{(G)} in the notation of \cite{schmidt_cap}), there are two different constructions of the compatible system of $\lambda$-adic Galois representations attached to $\pi$:
\begin{itemize}
\item The original construction, due to Laumon \cite{Laumon} and Weissauer \cite{Weissauer}, builds on previous work of Taylor \cite{taylor1993}, and works directly with a symplectic Shimura variety. The Galois representations are constructed from the \'etale cohomology of Siegel threefolds. The fact that the Galois representations are valued in $\Gf$ was proven by Weissauer in \cite{Weissauersymplectic}.
\item The second construction, due to Sorensen \cite{sorensen}, utilises the transfer map from $\Gf$ to $\GL_4$ in combination with Harris--Taylor's construction of Galois representations for automorphic representations of $\GL_4$, which uses unitary Shimura varieties \cite{harris-taylor}. Sorensen's costruction works in the more general setting of automorphic representations $\pi$ of $\Gf(\A_F)$, with $F$ totally real, so long as there exists a weak functorial lift of $\pi$ to $\GL_4$. For cohomological automorphic representations of $\Gf(\AQ)$, the existence of this lift is due to Weissauer \cite{Weissauersymplectic}*{Thm.\ 1}. Using this construction, Mok \cite[Thm.\ 3.5]{mok2014galois} proves local-global compatibility at ramified primes.
\end{itemize}

Irreducibility when $\l$ is sufficiently large is \cite[Thm.\ B]{Ramakrishnan}. Using \cite{pan}*{Thm.\ 1.0.4} in place of \cite{taylor-potential-modularity} in Ramakrishnan's proof, the condition on $\l$ can be taken to be $\l\ge 5$.
\end{proof}

\subsection{The case of non-cohomological weight}\label{section-galreps-noncohom}

The situation for low weight automorphic representations is much less comprehensive. Since the automorphic representations are non-cohomological, the associated Galois representations cannot be constructed directly from the \'etale cohomology of symplectic or unitary Shimura varieties. Instead, they are constructed as limits of cohomological Galois representations. The process of taking a limit of Galois representations loses information, in particular about local-global compatibility and geometricity at $\l$.

\begin{theorem}\label{non-cohom-gal-rep}
	Let $\pi$ be a cuspidal automorphic representation of $\Gf(\AQ)$ of weight $(k,2)$. Suppose that $\pi$ is not CAP or endoscopic. Let $S$ denote the set of primes at which $\pi$ is ramified.  Let $E$ denote the coefficient field of $\pi$. Then for every place $\lambda$ of $E$, of residue characteristic $\l$, there exists a continuous, semisimple, symplectic Galois representation
	\[\rho_\lambda\:\Ga\Q\to\Gf(\overline E_\lambda)\]
	that satisfies the following properties:
	\begin{enumerate}
		\item The representation is unramified at all primes $p\notin S\cup\{\l\}$.
		\item If $p\notin S\cup\{\l\}$, $a_p$ is the (suitably normalised) eigenvalue of the Hecke operator $T_p$ and $\chi$ is the Galois character associated to the central character of $\pi$, then 
		\[\Tr\rho_\lambda(\Frob_p)= a_p,\qquad
		\simil\rho_\lambda = \chi\epsilon_{\l}^{k_1+k_2-3},\]
		where $\epsilon_\l$ is the $\l$-adic cyclotomic character.
		\item The similitude character $\simil\rho_\lambda$ is odd and $\rho_\lambda\simeq\rho_\lambda\dual\tensor\simil\rho_\lambda$.
		\item Local-global compatibility is satisfied up to semisimplification:  for any prime $p\ne \l$, 
		\[\mathrm{WD}(\restr{\rho_\lambda}{\Qp})^{ss}\cong\mathrm{rec}_p(\pi_p\tensor|\simil|_p^{(3-k_1-k_2)/2})^{ss}.\]
		\item The Hodge--Tate--Sen weights $($i.e.\ the eigenvalues of the Sen operator$)$  of $\rho_\lambda|_{\Ql}$ are $\{0, 0, k-1, k-1\}.$
		\item If $\l\notin S$ and the roots of the $\l$-th Hecke polynomial of $\pi$ are pairwise distinct, then $\rho_\lambda|_{\Ql}$ is crystalline.
	\end{enumerate}
\end{theorem}

\begin{proof}
As in the cohomological case, there are two different constructions of the compatible system of $\lambda$-adic Galois representations attached to $\pi$. In both cases, $\rho_\lambda$ is constructed, via its pseudorepresentation, as a limit of cohomological Galois representations.
\begin{itemize}
\item The original construction, due to Taylor \cite{taylor1991galois}, uses the Hasse invariant to find congruences between the Hecke eigenvalue system of $\pi$ and mod $\l^n$ cohomological eigenforms $\pi_n$. The associated Galois pseudorepresentation is constructed as a limit of the Galois pseudorepresentations attached to the $\pi_n$. This construction is sufficient to prove the existence of the compatible system of Galois representations and parts $(i)$-$(iii)$ of the theorem.
\item A second construction, due to Mok \cite{mok2014galois}, extends the work of Sorensen \cite{sorensen} and constructs an eigencurve for $\Gf$. As in the cohomological case, this construction generalises to automorphic representations of $\Gf$ over totally real fields. However, the construction only requires the existence of a functorial lift from $\Gf$ to $\GL_4$ for cohomological forms, so the construction is still unconditional for automorphic representations of $\Gf(\AQ)$. Using this construction, Mok \cite[Thm.\ 3.5]{mok2014galois} proves local-global compatibility at ramified primes up to semisimplification. 
\end{itemize}

Part $(vi)$ is due to Jorza \cite[Thm.\ 4.1]{jorza2012}. Finally, the fact that the Galois representations are valued in $\Gf(\overline E_\lambda)$ is \Cref{image-Gsp4}.
\end{proof}

\begin{remark}
	While it should always be true that $\rho_\lambda|_{\Ql}$ is crystalline when $\l\notin S$, without the condition in part $(vi)$, we do not even know that $\rho_\lambda|_{\Ql}$ is Hodge--Tate.
\end{remark}

\begin{corollary}\label{bounded-conductor}
Let $\pi$ be a cuspidal automorphic representation of $\Gf(\AQ)$ of weight $(k,2)$, which is not CAP or endoscopic. Then there exists an integer $N$ such that, for all primes $\lambda$ of $E$, the Serre conductor of $\rho_\lambda$ divides $N$.
\end{corollary}

\begin{proof}
Fix a prime $\lambda$, and let $N_\rho$ be the Serre conductor of $\rho_\lambda$, where $\rho_\lambda$ is viewed as a representation valued in $\GL_4$. Let $S$ be the set of primes at which $\pi$ is ramified and let 
\[N_\pi = \prod_{p\in S}\cond(\mathrm{rec}_p(\pi_p\tensor|\simil|_p^{(3-k_1-k_2)/2}))\]
be the conductor of the transfer of $\pi$ to $\GL_4$. Since $S$ is finite, we can assume, without loss of generality, that $\l\notin S$. By definition,
\[N_\rho = \prod_{p\in S}\cond(\mathrm{WD}(\restr{\rho_\lambda}{\Qp})^{F\text{-}ss}).\]
A Weil--Deligne representation $(V, \rho, N)$ of $W_{\mathbb Q_p}$ has conductor
\[\cond(\rho)p^{\dim(V^I) - \dim(V^I_N)},\]
where $V^I$ is the subspace of $V$ fixed by the inertia group and $V^I_N= \ker(N)^I$. If $(V, \rho, N)$ is a Weil--Deligne representation, then, by definition, $(V, \rho, N)^{ss} = \rho^{ss}$. If, moreover, $(V, \rho, N)$ is Frobenius semisimple then $\rho^{ss} = \rho$, and it follows that
\[\cond(V, \rho, N) \mid \cond(\rho)p^{\dim(\rho)}.\]
Hence, $N_\rho$ divides $\prod_{p\in S}\cond(\mathrm{WD}(\restr{\rho_\lambda}{\Qp})^{ss})p^4.$
By part $(iv)$ of \Cref{non-cohom-gal-rep},
\[\prod_{p\in S}\cond(\mathrm{WD}(\restr{\rho_\lambda}{\Qp})^{ss})p^4=\prod_{p\in S}\cond(\mathrm{rec}_p(\pi_p\tensor|\simil|_p^{(3-k_1-k_2)/2})^{ss})p^4,\]
which divides
\[\prod_{p\in S}\cond(\mathrm{rec}_p(\pi_p\tensor|\simil|_p^{(3-k_1-k_2)/2}))p^4.\]
Since 
\[\prod_{p\in S}\cond(\mathrm{rec}_p(\pi_p\tensor|\simil|_p^{(3-k_1-k_2)/2}))p^4=N_{\pi}\prod_{p\in S}p^4,\]
we deduce that $N_\rho\mid N_\pi\prod_{p\in S}p^4$.
\end{proof}

\subsection{Galois representations valued in \texorpdfstring{$\Gf$}{GSp(4)}}\label{section-gsp4-valued}

The goal of this section is to prove the following theorem:

\begin{theorem}\label{image-Gsp4}
Let $\pi$ be a cuspidal automorphic representation of $\Gf(\AQ)$ of weight $(k,2)$, with associated $\lambda$-adic Galois representation $\rho_\lambda\:\Ga \Q\to\GL_4(\overline E_\lambda)$. Then $\rho_\lambda$ is isomorphic to a representation that factors through $\Gf(\overline E_\lambda)$.
\end{theorem}

The idea of the proof is to reformulate Taylor's original construction of $\rho_\lambda$, using V.\ Lafforgue's $G$-pseudorepresentations \cite{lafforgue2012chtoucas} in place of Taylor's pseudorepresentations \cite{taylor1991galois}. We are grateful to B. Stroh for providing a broad outline of the proof. The details of the proof are tangential to the remainder of the paper, so the reader should feel free to skip to the next section.

\begin{remark}
	This theorem has long been known to experts, however, prior to the appearance of this paper, no proof had appeared in the literature. Since the first appearance of this paper, an ostensibly simpler proof of \Cref{image-Gsp4} was given in \cite{voight-pacetti}*{Thm.\ 4.3.4}. Their proof is morally equivalent to the one we provide: the key fact in both proofs is that a representation valued in $\Gf$ is determined by its characteristic polynomial and its similitude. However, the argument given here is more robust, and does not require the low weight Siegel modular form to be a $p$-adic limit of cohomological eigenforms in characteristic $0$. Moreover, we expect that the arguments given here can be generalised to broader settings. For example, see \cite{weissberger} for a different application of these ideas.
\end{remark}

\subsubsection{Taylor's construction and the limitations of pseudorepresentations}

In \cite{taylor1991galois}, Taylor gives a blueprint for constructing Galois representations attached to low weight Siegel modular forms by utilising congruences with Siegel modular forms of cohomological weight. This subsection gives an overview of Taylor's construction.

Recall that $\pi$ is the cuspidal automorphic representation of $\Gf(\AQ)$ corresponding to a cuspidal Siegel modular eigenform of weight $(k,2)$ and level $\Gamma(N)$. Let $E$ be the finite extension of $\Q$ spanned by the Hecke eigenvalues of $\pi$, and fix a prime $\lambda$ of $E$ with residue characteristic $\l$. Following \cite{taylor1991galois}*{pp.\ 316}, let $\HH_p$ denote the Hecke algebra of $\Gf(\Z_p)$ bi-invariant functions $\Gf(\Qp)\to \Z$ that are supported in $\M_4(\Zp)$, and let $\HH^N = \widehat{\bigotimes}_{p\nmid N}\HH_p$. For each tuple $\vec{k} = (k_1, k_2)$ of weights, let $\T_{\vec{k}}$ denote the quotient of $\HH^N$ acting on the space of cuspidal Siegel modular forms of weight $\vec{k}$ and level $N$. Then $\T_{\vec{k}}\tensor\Q$ is a semisimple algebra.

Associated to $\pi$ is a character $\theta\:\HH^N\to\T_{(k, 2)}\to \O_{E_\lambda}$. Moreover, for each integer $i\ge 1$, the automorphic analogue of multiplying a classical form by the Hasse invariant \cite[Prop.\ 3]{taylor1991galois} gives a commutative diagram
\begin{center}
	\begin{tikzcd}
		\HH^N \arrow[r,"\theta"]\arrow[d,"\theta_i"]& \O_{E_\lambda}\arrow[d]\\
		\T_{\vec{k}_i}\tensor_\Z\O_{E_\lambda}\arrow[r,"r_i"]&\O_{E_\lambda}/\lambda^i
	\end{tikzcd}
\end{center}
where $\vec k_i = \br{k + a_\l\l^{i-1}(\l-1),2 + a_\l\l^{i-1}(\l-1)}$, with $a_\l\in\N$ a constant depending on $\l$. In the classical language, for each $i$, $\pi$ is congruent to a mod $\lambda^i$ eigenform of cohomological weight.

For every $i$, by \Cref{cohom-gal-rep} and the fact that $\T_{\vec{k}_i}\tensor\Q$ is semisimple, there is a finite extension $E_i/E_\lambda$ and a Galois representation
\[\rho_i\:\Ga \Q\to \Gf(\T_{\vec{k}_i}\tensor_\Z E_i)\]
such that $\Tr\rho_i(\Frob_p) = \theta_i(T_p)$ whenever $p\notin S_i$, for some finite set of places $S_i$. 

If we could compose $\rho_i$ with $r_i$ to construct a representation $\orho_i\:\Ga \Q\to \Gf(\O_{E_\lambda}/\lambda^i)$, then we would be able to construct $\rho_\lambda$ as the limit $\varprojlim_{i}\orho_i$. However, while $\Tr\rho_i(\Frob_p)\in \T_{\vec{k}_i}\tensor_\Z \O_{E_i}$ for all $p\notin S_i$, it is not necessarily true that $\rho_i$ can be chosen to be valued in $\Gf(\T_{\vec{k}_i}\tensor_\Z \O_{E_i})$. The solution to this problem is to work with pseudorepresentations. Associated to $\rho_i$ is a pseudorepresentation
\[T_i = \Tr\rho_i\:\Ga\Q\to\T_{\vec{k}_i}\tensor_\Z E_i\]
and, since $T_i(\Frob_p)\in\T_{\vec{k}_i}\tensor_\Z \O_{E_i}$ for all $p\notin S_i$ and since $\T_{\vec{k}_i}\tensor_\Z \O_{E_i}$ is flat, it is clear that
\[T_i \:\Ga\Q\to\T_{\vec{k}_i}\tensor_\Z \O_{E_i}\]
is valued in $\T_{\vec{k}_i}\tensor_\Z \O_{E_i}$. Composing with $r_i$, we obtain a pseudorepresentation
\[\overline{T}_i\:\Ga\Q\to\O_{E_i}/\lambda^i.\]
A computation shows that each $\overline T_i$ is in fact valued in $\O_{E_\lambda}/\lambda^i$ and that, for $i\ge m$, $\overline T_m \equiv\overline T_i\pmod{\lambda^m}$. Hence, there is a pseudorepresentation
\[T = \varprojlim_{i}\overline T_i\:\Ga \Q\to \O_{E_\lambda}\subset\overline E_\lambda.\]
It follows from the theory of pseudorepresentations \cite[Thm.\ 1]{taylor1991galois} that there is a semisimple Galois representation
\[\rho_\lambda\:\Ga\Q\to \GL_4(\overline E_\lambda)\]
associated to $T$, which is, by construction, the Galois representation associated to $\pi$.

Taylor's construction via pseudorepresentations shows that $\rho_\lambda$ is valued in $\GL_4(\overline E_\lambda)$, but is insufficient to show that the representation is isomorphic to one that is valued in $\Gf(\overline E_\lambda)$: taking the trace of $\rho_i$ `forgets' the fact that $\rho_i$ is symplectic. The proof of \Cref{image-Gsp4} follows the same structure as Taylor's proof, replacing pseudorepresentations with Lafforgue's $G$-pseudorepresentations. 

\subsubsection{Lafforgue pseudorepresentations}

In this subsection, we define Lafforgue pseudorepresentations and state their key properties. Most of these are lifted directly from \cite[Sec.\ 11]{lafforgue2012chtoucas} and \cite[Sec.\ 4]{bockle2016}.

Let $G$ be a split reductive group over $\Z$, and let $\Z[G^n]^{G}$ denote the ring of regular functions of $G^n$ that are invariant under conjugation by $G$. 

\begin{definition}
	Let $A$ be a topological ring, let $\Gamma$ be a topological group and let $C(\Gamma^n, A)$ denote the algebra of continuous functions $\Gamma^n\to A$. A (continuous) \emph{$G$-pseudorepresentation} $\Theta = (\Theta_n)_{n\ge 1}$ of $\Gamma$ over $A$ is a collection of continuous algebra homomorphisms
	$$\Theta_n\:\Z[G^n]^{G}\to C(\Gamma^n, A)$$
	for each integer $n\ge 1$, which are functorial in the following sense:
	
	\begin{enumerate}
		\item If $n, m\ge1$ are integers and if $\zeta\:\{1,\ldots,m\}\to\{1, \ldots,n \}$, then, for every $f\in \Z[G^m]^{G}$ and $\gamma_1,\ldots,\gamma_n\in \Gamma$, we have
		$$\Theta_n(f^\zeta)(\gamma_1,\ldots,\gamma_n) = \Theta_m(f)(\gamma_{\zeta(1)},\ldots, \gamma_{\zeta(m)}),$$
		where $f^\zeta(\gamma_1,\ldots,\gamma_n) := f(\gamma_{\zeta(1)},\ldots,\gamma_{\zeta(m)})$.
		\item For every integer $n\ge 1$, $f\in \Z[G^n]^{G}$ and $\gamma_1,\ldots,\gamma_{n+1}\in \Gamma$, we have
		$$\Theta_{n+1}(\hat{f})(\gamma_1,\ldots,\gamma_{n+1}) = \Theta_n(f)(\gamma_1, \ldots,\gamma_{n-1}, \gamma_n\gamma_{n+1}),$$
		where $\hat{f}(\gamma_1, \ldots, \gamma_{n+1}) :=f(\gamma_1, \ldots, \gamma_{n-1}, \gamma_n\gamma_{n+1})$. 
	\end{enumerate}
\end{definition}

For a category-theoretic interpretation of this definition, see \cite{Weidner}.

As with classical pseudorepresentations, we can change the ring $A$. The following facts are immediate from the definitions:

\begin{lemma}\label{pseudorep-compose}
	Let $A, A'$ be topological rings.
	\begin{enumerate}
		\item If $h\:A\to A'$ is a morphism of topological rings and if $\Theta = (\Theta_n)_{n\ge 1}$ is a $G$-pseudo\-representation over $A$, then $h_*(\Theta) = (h\circ\Theta_n)_{n\ge 1}$ is a $G$-pseudorepresentation over $A'$.
		\item Let $h\:A \hookrightarrow A'$ be an injection of topological rings and let $\Theta'$ be a $G$-pseudo\-representation over $A'$. Suppose that, for every $f\in \Z[G^n]^{G}$, $\Theta'_n(f) = h\circ g$ for some $g \in C(\Gamma^n, A)$. Then the collection $\Theta = (\Theta_n)_{n\ge 1}$ given by $\Theta_n(f) = g$ is a $G$-pseudo\-representation over $A$, and $\Theta' = h_*(\Theta)$.
	\end{enumerate}
\end{lemma}

The connection between $G$-pseudorepresentations and $G$-valued representations is encapsulated in the following lemma:

\begin{lemma}
	Let $\rho\:\Gamma\to G(A)$ be a continuous homomorphism. For each integer $n\ge1$, let
	\begin{align*}
		\Theta_n\:\Z[G^n]^{G}\to C(\Gamma^n, A)
	\end{align*}
	be given by
	$$\Theta(f)(\gamma_1,\ldots,\gamma_n) = f(\rho(\gamma_1),\ldots,\rho(\gamma_n)).$$
	Then the collection $(\Theta_n)_{n\ge1}$ is a $G$-pseudorepresentation, which we denote by $\Tr\rho$.
\end{lemma}

\begin{remark}
Fix an embedding $\iota\:G\hookrightarrow\GL_r$ for some $r$, and let $\chi$ denote the composition of $\iota$ with the usual trace function. Then $\chi\in \Z[G]^{G}$. Suppose that $\rho\:\Gamma\to G(A)$ is a representation with corresponding $G$-pseudorepresenation $\Tr\rho = (\Theta_n)_{n\ge1}$. Then $\Theta_1(\chi)\:\Gamma\to A$ is the classical pseudorepresentation associated to the representation $\iota\circ\rho$. Indeed, we have
$$\Theta_1(\chi)(\gamma) = \chi(\rho(\gamma)) = \Tr(\iota\circ\rho(\gamma)),$$
and the properties of this classical pseudorepresentation follow from the properties of $\Tr\rho$  \cite[Rmk.\ 11.8]{lafforgue2012chtoucas}.
\end{remark}

As in the case of classical pseudorepresentations, if $A$ is in fact an algebraically closed field, then every $G$-pseudorepresentation arises as the trace of a $G$-valued representation. 

\begin{theorem}[{\cite[Prop.\ 11.7]{lafforgue2012chtoucas}, \cite[Thm.\ 4.5]{bockle2016}}]\label{pseudorep-rep}
Let $A$ be an algebraically closed field and let $\Theta$ be a $G$-pseudorepresentation of $\Gamma$ over $A$. Then there is a completely reducible representation $\rho\:\Gamma\to G(A)$ such that $\Theta = \Tr(\rho)$.
\end{theorem}

	Being completely reducible generalises the notion of a $\GL_n$-representation being semisimple. Since we will not use this notion, we refer the reader to \cite[Def.\ 3.3]{bockle2016} for the definition.

\subsubsection{Lafforgue pseudorepresentations and Galois representations}
A key step in Taylor's construction is to show that the pseudorepresentation $T_i$ is valued in $\T_{\vec{k}_i}\tensor_\Z \O_{E_i}$. The following lemma will enable us to prove the analogue of this fact when using $\Gf$-pseudorepresentations in place of pseudorepresentations.

\begin{lemma}\label{invariant-hypothesis}
	Let $\chi_1, \ldots, \chi_r\in \Z[G]^{G}$. Suppose that, for each integer $n\ge 1$, $\Z[G^n]^{G}$ is generated by functions of the form
	\[(\gamma_1,\ldots,\gamma_n)\mapsto \chi_j(\gamma_{\zeta(1)}^{a_1}\gamma_{\zeta(2)}^{a_2}\cdots \gamma_{\zeta(m)}^{a_m}),\]
	where $1\le j\le r$, $m\ge 1$, $\zeta\:\{1,\ldots, m\}\to\{1,\ldots, n \}$ and $a_j \iZ$. Let
	\[\rho\:\Gamma\to G(A)\]
	be a continuous representation. Then $\Theta = \Tr\rho$ is completely determined by $\Theta_1(\chi_1), \ldots,\Theta_1(\chi_r)$.
\end{lemma}

\begin{proof}
	Let $\Theta$ be a $G$-pseudorepresentation, $n\ge1$ be an integer, $\gamma_1,\ldots,\gamma_n\in \Gamma$ and $f\in \Z[G^n]^{G}$. Since each $\Theta_n$ is an algebra homomorphism, we may assume that $f$ is of the form
	\[f\:(\gamma_1,\ldots,\gamma_n)\mapsto \chi(\gamma_{\zeta(1)}^{a_1}\gamma_{\zeta(2)}^{a_2}\cdots \gamma_{\zeta(m)}^{a_m}),\]
	where $\chi = \chi_j$ for some $j$, $\zeta\:\{1,\ldots, m\}\to\{1,\ldots, n \}$ and $a_j \iZ$. 
	
	First note that $f = g^\zeta$ where $g\in \Z[G^m]^{G}$ is given by
	\[g\:(\gamma_1,\ldots,\gamma_m)\mapsto\chi(\gamma_1^{a_1}\cdots \gamma_m^{a_m}).\]
	It follows that
	\[\Theta_n(f)(\gamma_1,\ldots,\gamma_n) =\Theta_m(g)(\gamma_{\zeta(1)},\ldots,\gamma_{\zeta(m)}).\]

	Since $\Theta = \Tr(\rho)$, we observe that
	\begin{align*}
		\Theta_m(g)(\gamma_{\zeta(1)},\ldots,\gamma_{\zeta(m)}) &= g\br{\rho(\gamma_{\zeta(1)}),\ldots,\rho(\gamma_{\zeta(m)})}\\
		&=\chi\br{\rho(\gamma_{\zeta(1)})^{a_1}\cdots\rho(\gamma_{\zeta(m)})^{a_m}}\\
		&= \Theta_m(g')(\gamma_{\zeta(1)}^{a_1},\ldots,\gamma_{\zeta(m)}^{a_m}),
	\end{align*}
	where
	\[g'\:(\gamma_1,\ldots,\gamma_m)\mapsto\chi(\gamma_1\cdots \gamma_m).\]	
	If $m\ge 2$, then $g' = \hat h$, where $h\in \Z[G^{m-1}]^{G}$ is given by
	\[h\:(\gamma_1,\ldots, \gamma_{m-1})\mapsto\chi(\gamma_1\cdots \gamma_{m-1}),\]
	so that 
	\[\Theta_m(g')(\gamma_{\zeta(1)}^{a_1},\ldots,\gamma_{\zeta(m)}^{a_m}) = \Theta_{m-1}(h)(\gamma_{\zeta(1)}^{a_1},\ldots,\gamma_{\zeta(m-1)}^{a_{m-1}}\gamma_{\zeta(m)}^{a_m}).\]
	It follows by induction on $m$ that 
	\[\Theta_n(f)(\gamma_1,\ldots,\gamma_n) = \Theta_1(\chi)(\gamma_{\zeta(1)}^{a_1}\cdots\gamma_{\zeta(m)}^{a_m}).\]
\end{proof}

\begin{example}
If $G = \GL_n$, then by work of Processi  \cites{Procesi-invariants, deconcini-procesi-book}, the ring $\Z[\GL_n^m]^{\GL_n}$ is generated by $\det\ii$ and the functions $X\mapsto s_i(X)$ that map $X\in\GL_n$ to the $i$-th coefficient of its characteristic polynomial, and $\Q[\GL_n^m]^{\GL_n}$ is generated just by the trace map and by $\det\ii$. In particular, a Lafforgue $\GL_n$-pseudorepresentation over a characteristic $0$ field is completely determined by its associated Taylor pseudorepresentation (c.f.\ \cite[Rmk.\ 11.8]{lafforgue2012chtoucas}).
\end{example}

In order to prove \Cref{image-Gsp4}, we show that $\Gf$ satisfies the conditions of \Cref{invariant-hypothesis}:

\begin{lemma}\label{Gsp4-invariants}
	For an element $X\in \Gf$, let $t^4 + \sum_{i=1}^{4}(-1)^is_i(X)t^{4-i}$ be its characteristic polynomial. Then $\Z[\Gf^m]^{\Gf}$ is generated by $\simil^{\pm1}$ and the functions $X\mapsto s_i(X)$, $i = 1,2$, where $s_i(X)$ is the $i$-th coefficient of the characteristic polynomial.
\end{lemma}

\begin{proof}
	Let $\M_4$ denote the algebra of $4\times 4$ matrices. If $K$ is any infinite field, then, by \cite{Zubkov}*{Thm.\ 1}, the ring $K[\M_4^m]^{\Sp_4} = K[\M_4^m]^{\Gf}$ is generated by functions of the form
	\[(X_1, \ldots, X_m)\mapsto s_i(Y_{j_1}\cdots Y_{j_s}),\]
	where each matrix $Y_i$ is either $X_i$ or its symplectic transpose $X_i^t$. By \cite{Zubkov}*{Prop.\ 3.2}, the ring $K[\Sp_4^m]^{\Sp_4}$ is generated by the images of these functions under the canonical map $K[\M_4^m]\to K[\Sp_4^m]$. If $X\in \Sp_4$, then, by definition, $X^t = X\ii$. It follows that $K[\Sp_4^m]^{\Sp_4}$ is generated by functions of the form
	\[(X_1, \ldots, X_m)\mapsto s_i(Y_{j_1}\cdots Y_{j_s}),\]
	where each matrix $Y_i$ is either $X_i$ or $X_i\ii$.
	
	Now, the natural grading on $\Z[\M_4^m]$ by multi-homogeneous degree restricts to a filtration on $\Z[\Sp_4^m]^{\Sp_4}$, and each filtered piece is a finite $\Z$-module. Hence, by the same argument as \cite{deconcini-procesi-book}*{15.2.1}, it follows that $\Z[\Sp_4^m]^{\Sp_4}$ is generated by functions of the above form as well.
	
	To conclude, we argue as in \cite{Zubkov}*{pp.\ 316}. The natural surjection $\Sp_4^m\times\GL_1^m\to \Gf^m$ induces an embedding 
	\[\Z[\Gf^m]^{\Gf}\hookrightarrow (\Z[\Sp_4^m]\tensor\Z[\GL_1^m])^{\Sp_4\times\GL_1}\simeq \Z[\Sp_4^m]^{\Sp_4}\tensor\Z[\GL_1^m].\]
	But, by the above computations, every element of $\Z[\Sp_4^m]^{\Sp_4}$ extends to an element of $\Z[\Gf^m]^{\Gf}$. Hence, the above map is surjective. It follows that $\Z[\Gf^m]^{\Gf}$ is generated by $\simil^{\pm1}$ and by functions of the above form, as required.
\end{proof}

\begin{proof}[Proof of Theorem $\ref{image-Gsp4}$]
We use the notation from the beginning of the section. Consider the $\Gf$-pseudorepresentation $\Theta^{(i)}=\Tr\rho_i$ associated to
	$$\rho_i\:\Ga\Q\to \Gf(\T_{\vec{k}_i}\tensor_\Z E_i).$$
	By Lemmas \ref{invariant-hypothesis} and \ref{Gsp4-invariants}, $\Tr\rho_i$ is completely determined by 
	$$\Theta^{(i)}_1(s_1) = T_i\:\Ga \Q\to \T_{\vec{k}_i}\tensor_\Z E_i,$$
	$$\Theta^{(i)}_1(s_2)\:\Ga \Q\to \T_{\vec{k}_i}\tensor_\Z E_i$$
	and
	$$\Theta^{(i)}_1(\simil^{\pm1})\:\Ga\Q\to \T_{\vec{k}_i}\tensor_\Z E_i.$$
	Since each of these maps factors through $\T_{\vec{k}_i}\tensor_\Z\O_{E_i}$, it follows from \Cref{pseudorep-compose}$(ii)$ that we can view each $\Theta^{(i)}$ as a $\Gf$-pseudorepresentation over $\T_{\vec{k}_i}\tensor_\Z\O_{E_i}$. By \Cref{pseudorep-compose}$(i)$ we may compose $\Theta^{(i)}$ with the map $r_i\:\T_{\vec{k}_i}\tensor_\Z\O_{E_i}\to \O_{E_i}/\lambda^i$ to produce a new $\Gf$-pseudorepresentation $\overline{\Theta}^{(i)}$ of $\Ga\Q$ over $\O_{E_i}/\lambda^i$. Since each $\Theta^{(i)}$ is determined by $\Theta_1^{(i)}(s_i),\ i=1,2$ and $\Theta_1^{(i)}(\simil^{\pm1})$, it follows that $\overline{\Theta}^{(i)}$ is too. Hence, the arguments of Taylor summarised above show that these maps actually land in $\O_{E_\lambda}/\lambda^i$, so that each $\overline\Theta^{(i)}$ is actually a $\Gf$-pseudorepresentation over $\O_{E_\lambda}/\lambda^i$. Therefore, we can form a $\Gf$-pseudorepresentation
		$$\Theta = \varprojlim_{i}\overline{\Theta}^{(i)}$$
	of $\Ga\Q$ over $\O_{E_\lambda}$. Finally, viewing $\O_{E_\lambda}$ as a subalgebra of $\overline E_\lambda$, we may view $\Theta$ as a $\Gf$-pseudorepresentation over $\overline E_\lambda$ and, by \Cref{pseudorep-rep}, there is a representation
	$$\rho\:\Ga\Q\to\Gf(\overline E_\lambda),$$
	such that $\Theta = \Tr(\rho)$. This is the Galois representation associated to $\pi$. Indeed, 
	$$\Theta_1(s_1) =T\:\Ga \Q\to \overline E_\lambda$$
	is exactly the classical pseudorepresentation constructed by Taylor.
\end{proof}

%% file: partial-irred.tex
\section{Restrictions on the decomposition of $\rho_\lambda$}\label{irred-section}

Recall that $\pi$ is a cuspidal automorphic representation of $\Gf(\AQ)$ of weight $(k, 2)$ in the notation of \Cref{sec:weight}, and that the weak functorial lift $\Pi$ of $\pi$ to $\GL_4$ exists and is cuspidal. The goal of this section is to prove the following theorem:

\begin{theorem}\label{thm:partial-irred}
	Keep the assumptions of Theorem $\ref{mainthm}$. For each prime $\lambda$, either:
	\begin{itemize}
		\item $\rho_\lambda$ is irreducible.
		\item $\rho_\lambda$ decomposes as a direct sum $\tau_1\+\tau_2$ of distinct, irreducible, two-dimensional representations, each with determinant $\simil\rho_\lambda$. Moreover, for each $i$, $\tau_i|_{\Ql}$ is Hodge--Tate with Hodge--Tate weights $\{0, k-1\}$.
	\end{itemize}
\end{theorem}

When $\pi$ is cohomological, \Cref{thm:partial-irred} is due to Weissauer \cite[Thm.\ II]{Weissauer} and Ramakrishnan \cite[Thm.\ A]{Ramakrishnan}. In the cohomological case, since $\rho_{\lambda}$ has distinct Hodge--Tate weights, the fact that the two representations are distinct and Hodge--Tate regular is obvious.

\subsection{Irreducibility without the Ramanujan bounds}

The proof of \cite[Thm.\ II]{Weissauer} makes crucial use of the Ramanujan bounds, which are not available in the non-cohomological case. In place of the Ramanujan bounds, we prove the following theorem, which requires only the Jacquet--Shalika bounds \cite{jacquetshalika1}, which are known for all cuspidal automorphic representations of $\GL_n$.

\begin{theorem}\label{thm:mock-ramanujan}
	Let $\rho\:\Ga\Q\to\GL_n(\Qlb)$ be a Galois representation. Assume that $\rho$ is automorphic: there exists a unitary cuspidal automorphic representation $\Pi$ of $\GL_n(\AQ)$ and an integer $w$ such that $L^*(\Pi, s-\frac w2) = L^*(\rho, s)$.
	
	Let $\tau$ be an $m$-dimensional subrepresentation of $\rho$, and suppose that $\det\tau|_{\Ql}$ is Hodge--Tate with Hodge--Tate weight $h$. Then
	\[{\left|h-\frac{mw}2\right|<\frac m2}.\]
\end{theorem}

\begin{remark}
	The generalised Ramanujan conjecture would give $h = \frac{mw}2$.
\end{remark}

\begin{proof}
	Fix an embedding $\Qb\hookrightarrow\C$ as well as a prime $p\ne \l$ at which both $\rho$ and $\Pi$ are unramified. Let $\alpha_1, \ldots, \alpha_n\iC$ be the Satake parameters of $\Pi_p$. Then, by assumption and via our fixed embedding, the eigenvalues of $\rho(\Frob_p)$ are $\alpha_1 p^{\frac w2}, \ldots, \alpha_n p^{\frac w2}$. 
	
	By \cite{jacquetshalika1}*{Cor.\ 2.5}, for each $i$, we have
	\[p^{\frac {w-1}2}<|\alpha_i p^{\frac w2}|< p^{\frac {w+1}2}. \]
	Since $\tau$ is an $m$-dimensional subrepresentation of $\rho$, it follows that $\det\tau(\Frob_p)$ is a product of $m$ of the eigenvalues of $\rho(\Frob_p)$. Thus
	\[p^{\frac {m(w-1)}2}<|\det\tau(\Frob_p)|<p^{\frac {m(w+1)}2}.\]
	On the other hand, the assumption on the Hodge--Tate weight of $\det\tau|_{\Ql}$ means that $\det\tau \simeq \chi\epsilon_\l^h$, where $\chi$ is a finite order character. Thus, $|\det\tau(\Frob_p)| = p^h$ and hence, ${\frac {m(w-1)}2}<h<{\frac {m(w+1)}2}$.
\end{proof}

\subsection{Proof of Theorem $\ref{thm:partial-irred}$}

\begin{lemma}\label{lem:aut-ind}
	Suppose that $\Pi$ is an automorphic induction. Then $\rho_\lambda$ is irreducible.
\end{lemma}

\begin{proof}
	Suppose that $\Pi$ is automorphically induced from an automorphic representation $\Pi'$ of $\GL_2(\A_F)$ for some quadratic extension $F/\Q$, and let $\rho_{\Pi', \lambda}\:\Ga K\to\GL_2(\overline E_\lambda)$ be the $\lambda$-adic Galois representation associated to $\Pi'$. It is well-known that $\rho_{\Pi', \lambda}$ is irreducible: when $F$ is real quadratic, this is due to  Ribet \cite{ribet-to-carayol}, while if $F$ is imaginary quadratic, Taylor \cite{taylorimaginary}*{Sec.\ 3} proved irreducibility in many cases. See \cite{weissthesis}*{Thm.\ 1.2.6} for a complete proof. 
	
	Since, by assumption, $\Pi$ is cuspidal, we have $\Pi' \not\cong (\Pi')^\sigma$, where $\sigma$ is the non-trivial element of $\Gal(F/\Q)$ \cite{ArthurClozel}. By local-global compatibility and the strong multiplicity one theorem for $\GL_n$, $\rho_{\lambda}\simeq \Ind_F^\Q\rho_{\Pi', \lambda}$ and $\rho_{\Pi', \lambda}\not\simeq\rho_{\Pi', \lambda}^\sigma$. It follows that $\rho_\lambda$ is irreducible.
\end{proof}

\begin{lemma}\label{lem:hodge--tate-subs}
	Suppose that $\tau$ is a proper subrepresentation of $\rho_\lambda$. Then $\tau$ is two-dimensional and $\tau|_{\Ql}$ is Hodge--Tate, with Hodge--Tate weights $\{0, k-1\}$.
\end{lemma}

\begin{proof}
	If $\tau$ is three-dimensional, then $\rho_\lambda$ also has a one-dimensional subrepresentation. Hence, we can assume that $\tau$ is $m$-dimensional, with $m = 1$ or $2$. By \Cref{non-cohom-gal-rep}, the Hodge--Tate--Sen weights of $\rho_\lambda|_{\Ql}$ are $\{0,0,k-1,k-1\}$. Thus, the Hodge--Tate--Sen weights $a_1,\ldots, a_m$ of $\tau|_{\Ql}$ are contained in $\{0,0,k-1,k-1\}$ and, by \Cref{thm:mock-ramanujan},  $a_1 + \cdots + a_m = \frac12 m(k-1)$. It follows that $m = 2$ and that $\tau|_{\Ql}$ has Hodge--Tate--Sen weights $\{0, k-1\}$. Since the Hodge--Tate--Sen weights are distinct integers, it follows that the Sen operator is semisimple \cite[Thm.\ 5.17]{mok2014galois}, and hence that $\tau|_{\Ql}$ is Hodge--Tate.
\end{proof}

\begin{lemma}\label{lem:no-finite}
	Suppose that $\rho_\lambda \simeq \tau_1\+\tau_2$, with $\tau_1, \tau_2$ irreducible and two-dimensional. Then there does not exist a finite order character $\chi$ with $\tau_1 \simeq\tau_2\tensor\chi$.
\end{lemma}

\begin{proof}
	Suppose that $\tau_1\simeq\tau_2\tensor\chi$ for a finite order character $\chi$. Then $\chi$ cuts out a finite cyclic extension $K$ of $\Q$, and $\rho_{\lambda}|_K\simeq\tau_1|_K\+\tau_1|_K$. We have
	\[\ext\rho_\lambda|_K \simeq (\tau_1|_K\tensor\tau_1|_K) \+ \det\tau_1|_K\+\det\tau_1|_K\simeq \Sym^2\tau_1|_K + \det\tau_1|_K^{\+3}.\]
	Now, $\det\tau_1|_K$ is Hodge--Tate at all places above $\l$ by \Cref{lem:hodge--tate-subs}, so, by class field theory, $L(\det\tau_1|_K, s)$ has meromorphic continuation to the whole of $\C$. Since $K/\Q$ is cyclic, by cyclic base change \cite{ArthurClozel} and by \cite{kim2003functoriality}, $\ext\rho_\lambda|_K$ is the Galois representation associated to an isobaric automorphic representation of $\GL_6(\A_K)$. Thus, $L(\ext\rho_\lambda|_K, s)$ has meromorphic continuation to the whole of $\C$. It follows that $L(\Sym^2\tau_1|_K, s)$  does too and, by twisting, so does $L(\Sym^2\tau_1\tensor\det\tau_1\ii|_K, s)$. 
	
	On the other hand,
	\[\rho_{\lambda}|_K\tensor\rho_{\lambda}|_K\dual\simeq(\tau_1\tensor\tau_1\dual)^{\+4}\simeq (\Sym^2\tau_1\tensor\det\tau_1\ii)^{\+4}\+{\mathbf{1}}^{\+4}.\]
	Let $\Pi$ be the transfer of $\pi$ to $\GL_4(\AQ)$. By assumption, $\Pi$ is cuspidal. Moreover, by \Cref{lem:aut-ind}, we may assume that $\Pi$ is not an automorphic induction. Hence, the cyclic base change $\Pi_K$ of $\Pi$ is also cuspidal \cite{ArthurClozel}*{Thm.\ 4.2} and, by \cite{jacquet1981euler}*{Prop.\ 3.6}, the $L$-function
	\[L^*(\Pi_K\tensor\Pi_K\dual,s) = L^*(\rho_{\lambda}|_K\tensor\rho_{\lambda}|_K\dual,s)\]
	has a simple pole at $s = 1$. Since $L(\mathbf{1}, s) = \zeta(s)$ also has a simple pole at $s = 1$, it follows that
	\[\ord_{s=1}L(\Sym^2\tau_1\tensor\det\tau_1\ii|_K, s)^4 = 3.\]
	But $L(\Sym^2\tau_1\tensor\det\tau_1\ii|_K, s)$ is meromorphic, so its order of vanishing at $s = 1$ must be an integer. This is a contradiction.
\end{proof}

\begin{proof}[Proof of Theorem $\ref{thm:partial-irred}$]
	After \Cref{lem:hodge--tate-subs}, it remains to show that, if $\rho_\lambda \simeq\tau_1\+\tau_2$, with $\tau_1, \tau_2$ irreducible and two-dimensional, then $\det\tau_1 \simeq \det\tau_2\simeq\simil\rho_\lambda$.
	
	Let $\chi = \det\tau_1\ii\tensor\simil\rho_\lambda$. By \Cref{lem:hodge--tate-subs}, $\chi|_{\Ql}$ has Hodge--Tate weight $0$, i.e.\ $\chi$ is a finite order character. We need to show that $\chi$ is the trivial character. First, since
	\[(\simil\rho_\lambda)^2 \simeq\det\rho_\lambda\simeq\det\tau_1\det\tau_2,\]
	it follows that $\chi\simeq\det\tau_2\tensor\simil\rho_\lambda\ii$. Moreover,
	\begin{align*}
				\tau_1\+\tau_2\simeq\rho_\lambda&\simeq \rho_\lambda\dual\tensor\simil\rho_\lambda\\
				&\simeq \br{\tau_1\dual\tensor\simil\rho_\lambda}\+\br{\tau_2\dual\tensor\simil\rho_\lambda}\\
				&\simeq \br{\tau_1\tensor\det\tau_1\ii\tensor\simil\rho_\lambda} \+\br{\tau_2\tensor\det\tau_2\ii\tensor\simil\rho_\lambda}\\
				&\simeq \br{\tau_1\tensor\chi}\+\br{\tau_2\tensor\chi\ii}.
			\end{align*}
		
		By Schur's lemma and by \Cref{lem:no-finite}, it follows that $\tau_i\tensor\chi\simeq\tau_i$ for each $i=1, 2$. Thus, $\rho_\lambda\tensor\chi\simeq\rho_\lambda$. Let $\Pi$ be the transfer of $\pi$ to $\GL_4$. By class field theory, we may view $\chi$ as a Hecke character, and it follows that $\Pi$ and $\Pi\tensor\chi$ have the same Hecke polynomials at almost all primes. By the strong multiplicity one theorem for $\GL_4$, we have $\Pi\simeq\Pi\tensor\chi$. Hence, by \cite{ArthurClozel}*{Thm.\ 4.2}, if $\chi$ is non-trivial, then $\Pi$ is an automorphic induction, in which case, $\rho_\lambda$ is already irreducible by \Cref{lem:aut-ind}. Hence, $\chi$ must be trivial.
\end{proof}

%% file: crystallinity.tex
\section{Crystallinity and distinctness of Satake parameters}\label{distinctness}

In the previous section, we proved that, if $\rho_\lambda$ is reducible, then it decomposes as a direct sum $\tau_1\+\tau_2$ of distinct, irreducible, odd, two-dimensional representations. In this section, we will apply that result to prove the following theorem:

\begin{theorem}\label{thm:distinctness}
	Keep the assumptions of Theorem $\ref{mainthm}$, and assume that the transfer $\Pi$ of $\pi$ to $\GL_4$ is not an automorphic induction. For a set of primes $p$ of Dirichlet density $1$, the roots of the $p$-th Hecke polynomial of $\pi$ are pairwise distinct.
\end{theorem}

Applying part \Cref{non-cohom-gal-rep}$(vi)$, i.e.\ \cite{jorza2012}*{Thm.\ 4.1}, we immediately deduce:

\begin{corollary}\label{cor:crystalline}
	The representation $\rho_\lambda|_{\Ql}$ is crystalline for all primes $\lambda\mid\l$ for a set of primes $\l$ of Dirichlet density $1$.
\end{corollary}

\begin{remark}
	If $\Pi$ is automorphically induced from an automorphic representation $\Pi'$ of $\GL_2(\A_F)$, then \Cref{cor:crystalline} still holds for $\pi$. Indeed, let $\rho_{\Pi', \lambda}$ be the $\lambda$-adic Galois representation attached to $\Pi'$. Then $\rho_\lambda\simeq \Ind_F^\Q\rho_{\Pi', \lambda}$. If $F$ is real quadratic, then $\Pi'$ arises from a cohomological Hilbert modular form. Hence, $\rho_{\Pi', \lambda}$ is crystalline, from which it is clear that $\rho_\lambda$ is too. If $F/\Q$ is imaginary quadratic, then one can prove \Cref{cor:crystalline} by slightly amending the arguments of this section. In both these cases, $\rho_\lambda$ is irreducible by \Cref{lem:aut-ind}.
\end{remark}

\subsection{The Lie algebra\texorpdfstring{ of $\rho_\lambda$}{}}

Fix a prime $\lambda$, and let $F/E_\lambda$ be a finite extension such that
\[\rho_\lambda\:\Ga\Q\to \Gf(F)\]
is defined over $F$. Let $G_\lambda$ be the Zariski closure of the image of $\rho_\lambda$ in ${\Gf}_{/F}$, let $G_\lambda^\circ$ be its identity connected component and let $\g_\lambda$ be its Lie algebra. In this subsection, we determine the Lie algebra $\g_\lambda$ and the group $G_\lambda$ when $\rho_\lambda$ is irreducible.

\begin{proposition}\label{thm:lie-alg}
	Assume that the weight of $\pi$ is not of the form $(2k-1, k+1)$ for some $k\ge 2$. If $\rho_\lambda$ is irreducible, then $\g_\lambda\simeq \gsp_4(F)$.
\end{proposition}

\begin{remark}
	If the weight of $\pi$ is of the form $(2k-1, k+1)$ for some $k\ge 2$---so, in particular, $\pi$ is cohomological---then we cannot rule out the case that $\g_\lambda\simeq\Sym^3\gl_2(F)$ in complete generality. However, when $\lambda\mid\l$ with $\l\ge 5$, then one can rule out this case by combining \cite{conti2016galois}*{Thm.\ 3.8} with \cite{pan}*{Thm.\ 1.0.4}.
\end{remark}

\begin{corollary}\label{cor:image-connected}
	Assume that the weight of $\pi$ is not of the form $(2k-1, k+1)$ for some $k\ge 2$. If $\rho_\lambda$ is irreducible, then $G_\lambda\simeq \Gf(F)$.
\end{corollary}

\begin{proof}
	 $G_\lambda$ is a Zariski closed subgroup of $\Gf(F)$. Moreover, $\Lie(G_\lambda)=\Lie(\Gf(F))$, so $G_\lambda$ is also a Zariski open subgroup of $\Gf(F)$. Since $\Gf(F)$ is Zariski-connected, it follows that $G_\lambda =\Gf(F)$.
\end{proof}

\begin{definition}
	Let $G$ be a group, and $k$ be a field. We say that a representation
	$$\rho\:G\to\GL_n(k)$$
	is \emph{Lie irreducible} if $\restr{\rho}H$ is irreducible for all finite index subgroups $H\le G$.
\end{definition}

\begin{definition}
	Let $G$ be a group, and $k$ be a field. We say that a representation
	$$\rho\:G\to\GL_n(k)$$
	is \emph{imprimitive} if it is absolutely irreducible, but there is a finite index subgroup $H$ of $G$ and a $\overline k$-representation $\tau$ of $H$ such that $\rho\tensor\overline k\simeq \Ind_{H}^{G}\tau$. Otherwise, we say that $\rho$ is \emph{primitive}.
\end{definition}

\begin{lemma}\label{lem:lie-irred}
	If $\rho_\lambda$ is irreducible, then it is Lie irreducible.
\end{lemma}

\begin{proof}
	First suppose that $\rho_\lambda \simeq \Ind_K^\Q\tau$ is imprimitive. We first note that we may assume that $K/\Q$ is quadratic. If not, then, by counting dimensions, $[K:\Q] = 4$. If $K$ contains a quadratic subextension $K'$, then $\rho_\lambda= \Ind_{{K'}}^{\Q}\br{\Ind_{{K'}}^{K}\tau}$. If $K$ does not contain a quadratic subfield, the proof of \cite[Lem.\ 5.3]{Gan} shows that $\rho_\lambda$ is induced from a different quadratic extension. It follows from Clifford theory that \[\rho_\lambda\simeq\rho_\lambda\tensor\chi_{K/\Q},\]
	where $\chi_{K/\Q}$ is the quadratic character that cuts out the extension $K/\Q$. Let $\Pi$ be the lift of $\pi$ to $\GL_4$ and recall that, by assumption, $\Pi$ is cuspidal and not an automorphic induction. By class field theory, we may view $\chi_{K/\Q}$ as a Hecke character, and the above isomorphism implies that $\Pi$ and $\Pi\tensor\chi_{K/\Q}$ are weakly equivalent. Hence, by the strong multiplicity one theorem for $\GL_4$,
	\[\Pi\simeq\Pi\tensor\chi_{K/\Q}.\]
	By \cite[Thm.\ 4.2]{ArthurClozel}, this isomorphism is equivalent to $\Pi$ being an automorphic induction, a contradiction.
	
	Hence, we may assume that $\rho_\lambda$ is imprimitive. By \cite{patrikis}*{Prop.\ 3.4.1}, we can write
	\[\rho_\lambda\simeq\tau\tensor\omega,\]
	where $\tau$ is a Lie irreducible representation of dimension $d$ with $d\mid 4$, and $\omega$ is an Artin representation of dimension $\frac 4d$. The fact that the Hodge--Tate--Sen weights of $\rho_\lambda$ are not all equal shows that $\rho_\lambda$ is not a twist of an Artin representation. Hence, $d\ne 1$. Suppose that $d=2$. If $\omega$ is imprimitive---say $\omega \simeq \Ind_K^\Q\chi$ for some quadratic extension $K/\Q$ and character $\chi$ of $\Gal(\Qb/K)$---then
	\[\rho_\lambda\simeq\Ind_K^\Q(\tau|_K\tensor\chi)\]
	is also imprimitve, a contradiction. Hence, we may assume that both $\tau$ and $\omega$ are primitive. It follows that $\Sym^2\tau$ and $\Sym^2\omega$ are both irreducible. Taking exterior squares, we find that
	$$\ext\rho_\lambda\simeq\ext(\tau\tensor\omega)\simeq\br{\ext\tau\tensor\Sym^2\omega}\+\br{\ext\omega\tensor\Sym^2\tau}$$
	does not contain a one-dimensional subrepresentation, contradicting the fact that $\rho_{\pi,\l}$ is symplectic.
	
	The only remaining possibility is that $d = 4$, i.e.\ that $\rho_\lambda$ is Lie irreducible.
\end{proof}

\begin{proof}[Proof of \Cref{thm:lie-alg}]
	Let $G_\lambda'$ be the commutator subgroup of $G_\lambda$ and let $\g_\lambda'$ be its Lie algebra. Since $\rho_\lambda$ is a semisimple representation, it follows that $G_\lambda$ is a reductive group, that $G_\lambda'$ is semisimple and, hence, that $\g_\lambda'$ is a semisimple Lie subalgebra of $\sp_4(F)$. Moreover, since the similitude of $\rho_\lambda$ does not have finite image, we have $\g_\lambda\simeq\g_\lambda'\+\gl_1$. 
	
	Fix an embedding $F\hookrightarrow\Qlb$. Then, by the classification of semisimple algebras, the Lie algebra $\g_\lambda'\tensor_F\Qlb$ is one of the following Lie algebras  \cite[9.3.1]{Hida}:
	
	\begin{enumerate}
		\item $\sp_4(\Qlb)$;
		\item $\sl_2(\Qlb)\times\sl_2(\Qlb)$;
		\item $\sl_2(\Qlb)$ embedded in a Klingen parabolic subalgebra;
		\item $\sl_2(\Qlb)$ embedded in a Siegel parabolic subalgebra;
		\item $\sl_2(\Qlb)$ embedded via the symmetric cube representation $\SL_2\to\Sp_4$;
		\item $\{1\}$.
	\end{enumerate}
	
	Let $d\rho_\lambda\:\g_\lambda\to \gsp_4(F)$ be the Lie algebra representation associated to the map $G_\lambda\hookrightarrow \Gf(F)$ 
	
	We need to show that $\g_\lambda'\simeq\sp_4(F)$. Since $\rho_\lambda$ is irreducible, by \Cref{lem:lie-irred}, $\rho_\lambda$ is Lie irreducible, which exactly says that $d\rho_\lambda$ is irreducible. Moreover, since $\g_\lambda$ is semisimple, we can write $\g = \g'\+\a$, where $\a$ is abelian. It follows that $d\rho_\lambda$ is irreducible if and only its restriction ot $\g_\lambda'$ is. Thus, $\g_\lambda'\tensor_F\Qlb$ cannot be as in cases $(ii)$, $(iii)$, $(iv)$ or $(vi)$. 
	
	Suppose that $\g_\lambda'\tensor_F\Qlb$ is as in case $(v)$, i.e.\ that $\g_\lambda'\tensor_F\Qlb\simeq\Sym^3\sl_2(\Qlb)$. Then $(G_\lambda^\circ)'\times_F\Qlb\simeq \Sym^3\SL_2(\Qlb)$. Since the similitude of $\rho_\lambda$ does not have finite image, it follows that $G^\circ\times_F\Qlb\simeq\Sym^3\GL_2(\Qlb)$. There is a finite Galois extension $K/\Q$ such that $\rho_\lambda(\Gal(\Qb/K))\sub G^\circ$. It follows that $\rho_\lambda|_K\simeq\Sym^3\tau$ for some two-dimensional representation $\tau$. If the Hodge--Tate--Sen weights of $\tau$ at any place $v\mid\infty$ are $\{a, b\}$, then the Hodge--Tate--Sen weights of $\rho_\lambda|_K$ at $v$ are $\{3a, 2a+b, a+2b, 3b\}$, which contradicts the assumption on the weights of $\pi$. 
	
	It follows that $\g_\lambda'\tensor_F\Qlb\simeq \sp_4(\Qlb)$. Now, $\g_\lambda'$ is a vector subspace of $\sp_4(F)$. The fact that $\g_\lambda'\tensor_F\Qlb\simeq\sp_4(F)\tensor_F\Qlb$ shows that $\g_\lambda'$ and $\sp_4(F)$ have the same dimension, and hence are equal.
\end{proof}

\subsection{Distinctness of Satake parameters}

We recall the following theorem, due to Rajan:

\begin{theorem}[\cite{Rajan}*{Thm.\ 3}]\label{thm:rajan}
	Let $F$ be a finite extension of $\Ql$ and let $\sG$ be an algebraic group defined over $F$. Let $X$ be an algebraic subscheme of $\sG$, defined over $F$, that is stable under the adjoint action of $\sG$. Let
	\[\rho\:\Ga \Q\to \sG(F)\]
	be a Galois representation, and let
	\[C = X(F)\cap \rho(\Ga \Q).\]
	Let $G$ denote the Zariski closure of $\rho(\Ga \Q)$ in $\sG_{/F}$, with identity connected component $G^\circ$ and component group $\Phi = G/G^\circ$. For each $\phi\in\Phi$, let $G^\phi$ denote the corresponding connected component, and let
	\[\Psi = \set{\phi\in\Phi: G^\phi\sub X}.\]
	Then the set of primes $p$ such that $\rho(\Frob_p)\in C$ has Dirichlet density $\frac{|\Psi|}{|\Phi|}$.
\end{theorem}

\begin{corollary}\label{cor:irred-distinct}
	Suppose that $\rho_\lambda$ is irreducible and not a symmetric cube lift. Then the roots of the $p$-th Hecke polynomial of $\pi$ are pairwise distinct for a set of primes $p$ of Dirichlet density $1$.
\end{corollary}

\begin{proof}
	Up to multiplication by a normalisation factor, if $\rho_\lambda$ is unramified at $p$, then the roots of the $p$-th Hecke polynomial are exactly the eigenvalues of $\rho_\lambda(\Frob_p)$. 
	
	Let $X\sub\Gf$ be the set of elements $g\in \Gf$ whose characteristic polynomials have indistinct eigenvalues. Then $X$ is a closed subscheme of $\Gf$---it is the vanishing set of the discriminant of the characteristic polynomial---that is stable under the conjugation action of $\Gf$. By \Cref{cor:image-connected}, the Zariski closure of $\rho_\lambda$ is $\Gf(F)$, which is connected, and clearly is not contained in $X$. The result follows from \Cref{thm:rajan}.
\end{proof}

\begin{lemma}\label{lem:sym-cube}
	Suppose that $\rho_\lambda\simeq\Sym^3\tau$ is a symmetric cube lift from an irreducible two-dimensional representation $\tau$. Then the roots of the $p$-th Hecke polynomial of $\pi$ are pairwise distinct for a set of primes $p$ of Dirichlet density $1$.
\end{lemma}

\begin{proof}
	By \Cref{lem:lie-irred}, $\rho_\lambda$ is Lie irreducible, and hence $\tau$ is as well. Arguing as in \Cref{thm:lie-alg} (or as in \cite{Ribet85}), we see that the Zariski closure of the image of $\tau$ is $\GL_2(F)$. Thus, as in \Cref{cor:irred-distinct}, it follows that the eigenvalues $\alpha_p, \beta_p$ of $\tau(\Frob_p)$ are distinct for a set of primes $p$ of Dirichlet density $1$. Moreover, by \cite{patrikis}*{Prop.\ 3.4.9}, $\alpha_p\ne\beta_p$ for a set of primes $p$ of Dirichlet density $1$. Lastly, if $\alpha_p=\zeta_3\beta_p$, where $\zeta_3$ is a cube root of unity, then $\Tr\Sym^2\tau(\Frob_p) = 0$. Thus, if $\alpha_p=\zeta_3\beta_p$ for a positive density of primes $p$, it follows from \cite{patrikis}*{Prop.\ 3.4.9} that $\Sym^2\tau$ is irreducible, but not Lie irreducible. But if $\Sym^2\tau|_K$ is reducible, then $\tau|_K$ is induced from a quadratic extension $L/K$, so $\tau|_L$ is reducible. Hence, $\rho_\lambda|_L$ is reducible, contradicting \Cref{lem:lie-irred}.
	
	It follows that the eigenvalues $\alpha_p^3, \alpha_p^2\beta_p, \alpha_p\beta_p^2, \beta_p^3$ of $\rho_\lambda(\Frob_p)$ are distinct for a set of primes $p$ of Dirichlet density $1$.
\end{proof}

\begin{lemma}\label{lem:one-lie-irred}
	Suppose that $\rho_\lambda\simeq\tau_1\+\tau_2$, with $\tau_1, \tau_2$ irreducible and two-dimensional. Then, up to reordering, $\tau_1$ is Lie irreducible. 
\end{lemma}

\begin{proof}
	Suppose that $\tau_1$ and $\tau_2$ are both irreducible, but not Lie irreducible. Then, for each $i$, there exists a quadratic extension $K_i/\Q$ such that $\tau_i|_{K_i}$ is reducible. Let $K = K_1\cdot K_2$. Then $K$ is a solvable extension of $\Q$. Recall that there is a cuspidal automorphic representation $\Pi$ of $\GL_4(\AQ)$ associated to $\pi$. By \cite{ArthurClozel} and the assumption that $\Pi$ is not an automorphic induction, the solvable base change $\Pi_K$ of $\Pi$ is cuspidal. On the other hand, $\rho_\lambda|_K$ is a direct sum of four Hodge--Tate characters, $\chi_1\+\chi_2\+\chi_3\+\chi_4$. By class field theory, we may view the characters $\chi_i$ as Hecke characters. By the strong multiplicity one theorem for $\GL_4$, it follows that $\Pi_K$ is isomorphic to the isobaric sum of four Hecke characters, contradicting the fact that $\Pi_K$ is cuspidal.
\end{proof}

\begin{proof}[Proof of \Cref{thm:distinctness}]
	After \Cref{thm:partial-irred}, \Cref{cor:irred-distinct} and \Cref{lem:sym-cube}, it remains to consider the case that $\rho_\lambda\simeq\tau_1\+\tau_2$, with $\tau_1, \tau_2$ distinct, irreducible two-dimensional Hodge--Tate representations, each with Hodge--Tate weights $\{0,k-1\}$ and with determinant $\simil\rho_\lambda$.
	
	For each $i=1, 2$ and for each prime $p$ at which $\rho_{\lambda}$ is unramified, let $\alpha_{p, i}, \beta_{p,i}$ be the roots of the characteristic polynomial of $\tau_i(\Frob_p)$. By \Cref{lem:one-lie-irred}, we may assume that $\tau_1$ is Lie irreducible. Hence, as in the proof of \Cref{lem:sym-cube}, we see that $\alpha_{p,1}\ne \beta_{p,1}$ for a set of primes $p$ of Dirichlet density $1$. 
	
	If $\tau_2$ is Lie irreducible, then the same argument shows that $\alpha_{p,1}\ne \beta_{p,2}$ for a set of primes $p$ of Dirichlet density $1$. Otherwise, since $\tau_2$ is not an Artin representation, it follows that there is a quadratic extension $K/\Q$ and a character $\chi$ of $\Gal(\Qb/K)$ such that
	\[\tau_2 \simeq\Ind_K^\Q(\chi).\]
	If $p$ is a prime that is inert in $K$, then $\Tr\tau_2(\Frob_p)= \alpha_{p,1}+ \beta_{p,2}= 0$, from which it follows that $\alpha_{p,1}\ne \beta_{p,2}$. If $p$ splits as $vv^c$ in $K$, then the eigenvalues of $\tau_2(\Frob_p)$ are $\chi(\Frob_v)$ and $\chi(\Frob_{v^c}) = \chi^c(\Frob_v)$, where $c$ is the non-trivial element of $\Gal(K/\Q)$. Note that $\chi\not\simeq\chi^c$, since $\tau_2$ is irreducible. Moreover, since $\tau_2$ is Hodge--Tate regular, it follows that $K/\Q$ is imaginary quadratic, that $\chi$ and $\chi^c$ have infinite image and that they have different Hodge--Tate weights at the two complex places. If $\chi(\Frob_v) =\chi^c(\Frob_v)$ for a positive density of primes $p$, then, by \cite{Rajan}*{Thm.\ 2}, there exists a finite order character $\omega$ such that $\chi\simeq\chi^c\tensor\omega$, which contradicts the fact that $\chi$ and $\chi^c$ have distinct Hodge--Tate weights.
	
	It follows that for a set of primes $p$ of Dirichlet density $1$, we have $\alpha_{p, i}\ne \beta_{p, i}$ for each $i = 1,2$.  If the eigenvalues of $\rho_\lambda(\Frob_p)$ are indistinct for a positive density of primes $p$, relabelling, we may assume that $\alpha_{p,1} = \alpha_{p,2}$. Since $\det\tau_1\simeq\det\tau_2$, it follows that $\beta_{p,1} = \beta_{p,2}$ as well. Hence, for a set of primes $p$ of positive density, we have $\Tr\tau_1(\Frob_p) = \Tr\tau_2(\Frob_p)$. By \cite{Rajan}*{Thm.\ 2}, there exists a finite order character $\chi$ such that $\tau_1\simeq\tau_2\tensor\chi$, contradicting \Cref{lem:no-finite}.
\end{proof}

%% file: full-irred.tex
\section{Proof of Theorem $\ref{mainthm}$}\label{full-irred}

\begin{proof}[Proof of  Theorem $\ref{mainthm}$]
	After \Cref{cor:crystalline}, it remains to show that, if $\lambda\mid\l$ with $\l\ge 5$ and if $\rho_\lambda|_{\Ql}$ is de Rham, then $\rho_\lambda$ is irreducible.
	
	Suppose that $\rho_\lambda$ is reducible. Then, by \Cref{thm:partial-irred}, we can write $\rho_\lambda\simeq\tau_1\+\tau_2$ where the representations $\tau_1, \tau_2$ are distinct, two-dimensional, irreducible, Hodge--Tate regular and odd. Since $\rho_\lambda|_{\Ql}$ is de Rham, so are $\tau_1|_{\Ql}$ and $\tau_2|_{\Ql}$. Thus, since $\l\ge 5$, by \cite{pan}*{Thm.\ 1.0.4}, there exist distinct cuspidal automoprhic representations $\pi_1, \pi_2$ of $\GL_2(\AQ)$ such that, for each $i$, $\tau_i$ is the $\lambda$-adic Galois representation associated to $\pi_i$.
	
	Consider the representation
	\[\rho_\lambda\tensor\rho_\lambda\dual\simeq (\tau_1\tensor\tau_1\dual)\+(\tau_1\tensor\tau_2\dual)\+(\tau_2\tensor\tau_1\dual)\+(\tau_2\tensor\tau_2\dual).\]
	Since $\rho_\lambda$, $\tau_1$ and $\tau_2$ are automorphic, we obtain an equality of (partial) $L$-functions
	\[L^*(\pi\tensor\pi\dual,s) = L^*(\pi_1\tensor \pi_1\dual, s)L^*(\pi_1\tensor \pi_2\dual, s)L^*(\pi_2\tensor \pi_1\dual, s)L^*(\pi_2\tensor \pi_2\dual, s).\]
	Since the transfer of $\pi$ to $\GL_4$ is cuspidal, by \cite{jacquet1981euler}*{Prop.\ 3.6}, the left hand side has a simple pole at $s=1$. Similarly, for each $i = 1,2$, $\ord_{s=1}L^*(\pi_i\tensor\pi_i\dual,s) = -1$.
	However, by \cite{jacquet1981euler}*{Thm.\ 3.7}, the $L$-functions $L^*(\pi_1\tensor \pi_2\dual, s)$ and $L^*(\pi_2\tensor \pi_1\dual, s)$ are non-zero at $s = 1$. It follows that the right hand side has a pole of order $2$ at $s = 1$, a contradiction. Hence, $\rho_\lambda$ is irreducible.
\end{proof}

%% file: residual-irred.tex
\section{Residual irreducibility and the image of Galois}\label{residual-irred-section}

In this section, we prove \Cref{residual-irred-intro}. Our arguments generalise those of \cites{Dieulefait2002maximalimages, Dieulefait-endoscopy, DZ} to the case that $\pi$ is non-cohomological. Moreover, \Cref{lem:no-even} allows us to strengthen the results of \cite{DZ} even in the cohomological case (see \cite[Rmk.\ 3.4]{DZ}). Although the results in this section apply to automorphic representations of arbitrary weight, for ease of notation, we will assume that $\pi$ has non-cohomological weight $(k,2)$, for some integer $k\ge 2$.

Our key tool is following proposition, which is a simple consequence of Fontaine--Laffaille theory. Recall that $\LL$ is the set of primes $\lambda$ for which $\rho_\lambda|_{\Ql}$ is crystalline. By \Cref{mainthm}, $\LL$ contains all primes $\lambda\mid\l$ for a set of rational primes $\l$ of Dirichlet density $1$.

\begin{proposition}\label{inertia-types}
	Suppose that $\lambda\in\LL$ has residue characteristic $\l>k$. Then we have the following possibilities for the action of the inertia group $I_\l$ at $\l$ $($c.f.\ \cite{Dieulefait2002explicit}$)$:
	
	\[
	\begin{pmatrix}
		1 & * & * & *\\
		0 & \overline\epsilon_\l^{k-1} & * & *\\
		0&0&1&*\\
		0&0&0&\overline\epsilon_\l^{k-1}
	\end{pmatrix},
	\begin{pmatrix}
		\psi_2^{k-1} & 0 & * & *\\
		0 & \psi_2^{\l(k-1)} & * & *\\
		0&0&1&*\\
		0&0&0&\overline\epsilon_\l^{k-1}
	\end{pmatrix},
	\]\[
	\begin{pmatrix}
		\psi_2^{k-1} & 0 & * & *\\
		0 & \psi_2^{\l(k-1)} & * & *\\
		0&0&\psi_2^{\l(k-1)}&0\\
		0&0&0&\psi_2^{k-1}
	\end{pmatrix},
	\begin{pmatrix}
		\psi_4^{(\l+\l^2)(k-1)} & 0 & 0 & 0\\
		0 & \psi_4^{(\l^2+\l^3)(k-1)} & 0 & 0\\
		0&0&\psi_4^{(\l^3+1)(k-1)}&0\\
		0&0&0&\psi_4^{(1+\l)(k-1)}
	\end{pmatrix},
	\]
	where $\overline\epsilon_\l$ is the mod $\l$ cyclotomic character, and $\psi_i$ is the fundamental character of level $i$.
\end{proposition}

\begin{lemma}\label{lem:no-resid-one-dim}
	Suppose that, for infinitely many $\lambda\in\LL$, $\orho_\lambda$ contains a subrepresentation $\overline\tau_\lambda$. Then for all but finitely many such $\lambda$, $\overline{\tau}_\lambda$ is irreducible and two-dimensional, and $\det\overline\tau_\lambda|_{I_\l} = \overline\epsilon_\l^{k-1}$.
\end{lemma}

\begin{proof}
	First suppose that, for infinitely many $\lambda\in\LL$, $\overline\tau_\lambda$ is one-dimensional. By \Cref{inertia-types}, it follows that, for almost all such $\lambda$, $\overline\tau_\lambda\simeq\overline\chi_\lambda\overline\epsilon_\l^{n_\lambda}$, where $\overline\chi_\lambda$ is unramified at $\l$ and $n_\lambda = 0$ or $k-1$. By \Cref{bounded-conductor}, the conductor of $\overline\chi_\lambda$ is bounded independently of $\lambda$. Since there are only finitely many characters of bounded conductor, there exists a character $\chi\:\Ga\Q\to\Qb\t$ such that, for infinitely many $\lambda$, $\overline\tau_\lambda = \chi\epsilon_\l^{n}\pmod\lambda$, where $n = 0$ or $k-1$ is independent of $\lambda$. 
	
	Fix a prime $p$ at which $\pi$ is unramified. Then, for infinitely many $\lambda$, $\rho_\lambda(\Frob_p)$ has an eigenvalue $\alpha_p$ such that $\alpha_p \equiv \chi\epsilon_\l^n(\Frob_p) = \chi(p)p^n\pmod \lambda$. It follows that $\alpha_p = \chi(p)p^n$. Since $n\ne \frac{k-1}2$, this contradicts the Jacquet--Shalika bounds \cite{jacquetshalika1}*{Cor.\ 2.5}, as in \Cref{thm:mock-ramanujan}.
	
	Hence, if $\orho_\lambda$ contains a subrepresentation $\overline\tau_\lambda$ for infinitely many $\lambda\in\LL$, then $\overline\tau_\lambda$ must be two-dimensional for all but finitely many such $\lambda$. If $\det\overline\tau_\lambda|_{I_\l} \ne\overline\epsilon_\l^{k-1}$, then, by \Cref{inertia-types}, we must have $\det\overline\tau_\lambda =\overline\chi_\lambda\overline\epsilon_\l^{n_\lambda}$, where $\overline\chi_\lambda$ is unramified at $\l$ and $n_\lambda = 0$ or $2(k-1)$. As above, it follows that for a fixed prime $p$, $\rho_\lambda(\Frob_p)$ has two eigenvalues $\alpha_p, \beta_p$ that satisfy $|\alpha_p\beta_p| = p^n$ with $n = 0$ or $2(k-1)$, which contradicts the Jacquet--Shalika bounds.
\end{proof}

\begin{lemma}\label{lem:no-even}
	Suppose that, for infinitely many $\lambda\in\LL$, $\orho_\lambda\simeq\overline\tau_{\lambda}\+\overline\tau_\lambda'$, with $\overline\tau_{\lambda}, \overline\tau_{\lambda}'$ irreducible and two-dimensional.  Then for all but finitely many such $\lambda$, $\overline\tau_{\lambda}, \overline\tau_{\lambda}'$ are odd.
\end{lemma}

\begin{proof}
	Write $\overline\chi_\lambda = \simil\orho_\lambda\tensor\det\overline\tau_\lambda\ii$ and suppose that $\det\overline\tau_\lambda$ is even. Since $\simil\orho_\lambda$ is odd, it follows that $\overline\chi_\lambda$ is non-trivial.  Moreover, since $\det\orho_\lambda \simeq (\simil\orho_\lambda)^2\simeq\det\overline\tau_\lambda\det\overline\tau_\lambda'$, it follows that $\overline\chi_\lambda= \det\overline\tau_\lambda'{}\simil\orho_\lambda\ii$. 
	
	Now, by \Cref{lem:no-resid-one-dim}, we may assume that $\overline\chi_\lambda$ is unramified at $\l$ and, by \Cref{bounded-conductor}, the conductor of $\overline\chi_\lambda$ is bounded independently of $\lambda$. Hence, there exists a Dirichlet character $\chi$ such that $\overline\chi_\lambda = \chi\pmod\lambda$ for infinitely many $\lambda$.
	
	By duality, we see that
	\[\overline\tau_{\lambda}\+\overline\tau_\lambda'\simeq \orho_\lambda\simeq\orho_\lambda\dual\tensor\simil\orho_\lambda\simeq (\overline\tau_{\lambda}\dual\tensor\simil\orho_\lambda)\+(\overline\tau_\lambda'{}\dual\tensor\simil\orho_\lambda) \simeq (\overline\tau_\lambda\tensor\overline\chi_\lambda)\+(\overline\tau_\lambda'\tensor\overline\chi_\lambda\ii).\]
	By Schur's lemma, it follows that either $\overline\tau_\lambda\simeq \overline\tau_\lambda\tensor\overline\chi_\lambda$ and $\overline\tau_\lambda'\simeq \overline\tau_\lambda'\tensor\overline\chi_\lambda$ or $\overline\tau_\lambda'\simeq \overline\tau_\lambda\tensor\overline\chi_\lambda$.
	
	In the first case, $\orho_\lambda\tensor\overline\chi_\lambda\simeq\orho_\lambda$. If this case occurs for infinitely many $\lambda$, then $\rho_\lambda\simeq\rho_\lambda\tensor\chi$. Thus, $\rho_\lambda$ is not Lie irreducible, contradicting \Cref{lem:lie-irred}.
	
	In the second case, $\orho_\lambda|_{\ker\chi}\simeq \overline{\tau}_\lambda|_{\ker\chi}^{\+2}$. Thus, if $p$ splits in the number field $\Qb^{\ker\chi}$, then the eigenvalues of $\orho_\lambda(\Frob_p)$ are indistinct. If this case occurs for infinitely many $\lambda$, then for the positive density of primes $p$ that split in $\Qb^{\ker\chi}$, the $p$-th Hecke polynomial has indistinct roots modulo $\lambda$ for infinitely many $\lambda$. Hence, for a positive density of primes $p$, the $p$-th Hecke polynomial has indistinct roots in $\Qb$, contradicting \Cref{thm:distinctness}.
\end{proof}

\begin{proof}[Proof of \Cref{residual-irred-intro}]
	We first show that $\orho_\lambda$ is irreducible for all but finitely many $\l\in\LL$. After the previous lemmas, it remains to consider the case that, for infinitely many $\l\in\LL$, $\orho_\lambda\simeq\overline\tau_{\lambda}\+\overline\tau_\lambda'$, with $\overline\tau_{\lambda}, \overline\tau_{\lambda}'$ irreducible, two-dimensional and odd. In this case, we can argue as in \cite{Dieulefait2002maximalimages} and apply Serre's conjecture \cite{khare2009serre}.
	
	By \Cref{bounded-conductor}, the conductors of $\overline\tau_{\lambda}, \overline\tau_{\lambda}'$ are bounded independently of $\lambda$, and by \Cref{lem:no-resid-one-dim}, they each have Serre weight $k$. Thus, there is an integer $N$, independent of $\lambda$, and modular forms $f_\lambda, f_\lambda'\in S_k(\Gamma_1(N))$ with associated residual representations $\overline\tau_{\lambda}, \overline\tau_{\lambda}'$. 
	
	If this case occurs for infinitely any $\lambda$, then, since $S_k(\Gamma_1(N))$ is finite dimensional, there exist fixed modular forms $f, f'\in S_k(\Gamma_1(N))$ such that $f = f_\lambda$ and $f'=f'_\lambda$ for infinitely many $\lambda$. Thus, for infinitely many $\lambda$ and for almost all $p$, $\Tr\rho_\lambda(\Frob_p)\equiv a_f(p) + a_{f'}(p)\pmod\lambda$ and, since this congruence holds for infinitely many $\lambda$, it must be an equality. By the Chebotarev density theorem and the fact that a semisimple representation in charactersitic $0$ is determined by its trace, it follows that $\rho_\lambda$ is reducible for all $\lambda$, contradicting \Cref{mainthm}. Hence, $\orho_\lambda$ is irreducible for all but finitely many $\lambda\in\LL$.
	
	The remainder of the proof of \Cref{residual-irred-intro} is exactly the same as for the cohomological case \cite[3.2-3.5]{DZ}. By the classification of the maximal subgroups of $\GSp_4(\F_{\l^n})$ \cite{mitchell1914subgroups}, if $\orho_\lambda$ is irreducible and does not contain $\Sp_4(\Fl)$, then one of the following cases must hold:
	\begin{enumerate}
		\item The image contains a reducible index two subgroup, i.e.\ $\orho_\lambda$ is induced from a quadratic extension;
		\item $\orho_\lambda$ is isomorphic to the symmetric cube of a two-dimensional representation;
		\item The image is a small exceptional group.
	\end{enumerate}
	
	By using the description of the image of inertia as in \Cref{inertia-types}, Dieulefait--Zenteno show that each of these cases can only occur for finitely many $\l\in\LL$.
\end{proof}

\section*{Acknowledgements}

I would like to thank my Ph.D. supervisor Tobias Berger for suggesting this problem, and for his constant help, guidance and support throughout my Ph.D. I am also grateful to Luis Dieulefait for a helpful discussion, which helped to shape the direction of the paper, and to Neil Dummigan and Toby Gee for their valuable feedback on my Ph.D. thesis. Thanks also to Adel Betina, Andrea Conti, Frazer Jarvis, Jayanta Manoharmayum, Mohamed Moakher, Vincent Pilloni, Ciaran Schembri, Haluk \c{S}eng\"un, Beno\^it Stroh and Jacques Tilouine for helpful conversations, correspondences and feedback. Lastly, I would like to thank the anonymous referee for the helpful (and timely) comments and corrections they made for the final version of this paper.